\documentclass{article}
\def\letitre{Deformations of nearly K\"ahler instantons}   %
\title{\letitre}
\author{Benoit Charbonneau and Derek Harland}
\newcommand{\addressBenoit}{Department of Pure Mathematics, 
				University of Waterloo, 
				200 University Avenue West, Waterloo, Ontario N2L 3G1, Canada.
				and
				Perimeter Institute for Theoretical Physics, Waterloo, Ontario N2L 2Y5, Canada.}    
\newcommand{\emailBenoit}{benoit@alum.mit.edu}          
\newcommand{\addressDerek}{School of Mathematics, University of Leeds, LS2 9JT, UK.}
\newcommand{\emailDerek}{d.g.harland@leeds.ac.uk}

\usepackage{amssymb,amsmath,amsthm,enumerate}
\usepackage[usenames,dvipsnames]{color}
\usepackage[colorlinks, citecolor=blue, linkcolor=blue, urlcolor=Maroon, filecolor=Maroon]{hyperref}
\usepackage{url}

\newcommand{\RR}{{\mathbb{R}}}
\newcommand{\CC}{{\mathbb{C}}}
\newcommand{\C}{\CC}

\newcommand{\HH}{{\mathbb{H}}}

\renewcommand{\gg}{\mathfrak{g}}
\newcommand{\g}{\gg}
\newcommand{\hh}{\mathfrak{h}}
\newcommand{\mm}{\mathfrak{m}}
\newcommand{\kk}{\mathfrak{k}}
\newcommand{\su}{\mathfrak{su}}
\newcommand{\so}{\mathfrak{so}}
\newcommand{\sll}{\mathfrak{sl}}
\newcommand{\uu}{\mathfrak{u}}
\renewcommand{\sp}{\mathfrak{sp}}

\newcommand{\SU}{\mathrm{SU}}
\newcommand{\U}{\mathrm{U}}
\newcommand{\Sp}{\mathrm{Sp}}

\newcommand{\Tr}{\mathrm{Tr}}
\newcommand{\tr}{\Tr}
\newcommand{\End}{\mathrm{End}}
\newcommand{\Sym}{\mathrm{Sym}}
\newcommand{\Cl}{\mathrm{Cl}}
\newcommand{\vect}[1]{\begin{bmatrix}#1\end{bmatrix}}
\newcommand{\diag}{\mathrm{diag}}
\newcommand{\Span}{\mathrm{span}}

\newcommand{\BB}{\mathcal{B}}

\newcommand{\Ad}{\mathrm{Ad}}
\newcommand{\ad}{\mathrm{ad}}
\newcommand{\Cas}{\mathrm{Cas}}

\newcommand{\Vol}{\mathrm{Vol}_g}
\newcommand{\Scal}{\mathrm{Scal}}
\newcommand{\Ric}{\mathrm{Ric}}
\newcommand{\PPP}{\mathcal{P}}
\newcommand{\dd}{{\rm d}}

\newcommand{\ii}{\mathrm{i}}
\newcommand{\jj}{\mathrm{j}}
\newcommand{\kkk}{\mathrm{k}}

\newcommand{\sfrac}[2]{{\textstyle\frac{#1}{#2}}}

\theoremstyle{plain}
\newtheorem{theorem}{Theorem}
\newtheorem{lemma}{Lemma}
\newtheorem{proposition}{Proposition}
\newtheorem{corollary}{Corollary}
\theoremstyle{definition}

\usepackage{color}

\usepackage[margin=2.5cm,includehead]{geometry}
\usepackage{fancyhdr}
\fancypagestyle{plain}{%
   \fancyhead[LO]{}
   \fancyhead[CO]{}\fancyhead[RO]{}
   \fancyfoot[LO]{Note: This ArXiv version uses the numbering of the
                         version published in \textbf{Communications in Mathematical Physics}, except for the pages.  This article is published with open access at SpringerLink.com under the \href{http://dx.doi.org/10.1007/s00220-016-2675-y}{doi 10.1007/s00220-016-2675-y}}
}

\begin{document}

\date{8 June 2016.}
\maketitle
\begin{abstract}
	We formulate the deformation theory for instantons on nearly K\"ahler six-manifolds using spinors and Dirac operators.  Using this framework we identify the space of deformations of an irreducible instanton with semisimple structure group with the kernel of an elliptic operator, and prove that abelian instantons are rigid.  As an application, we show that the canonical connection on three of the four homogeneous nearly K\"ahler six-manifolds $G/H$ is a rigid instanton with structure group $H$.  In contrast, these connections admit large spaces of deformations when regarded as instantons on the tangent bundle with structure group $\SU(3)$.
\end{abstract}
\footnotetext{The authors can be reached respectively at [\emailBenoit\ \addressBenoit], [\emailDerek\ \addressDerek]}

\section{Introduction}

Instantons are connections whose curvature solves a certain linear algebraic equation.  Although instantons were first introduced in dimension four, the study of instantons on manifolds of dimension greater than four has a long history
\cite{Acharya1997,Baraglia-Hekmati,Baulieu1998,Capria-Salamon,Cherkis-octonions-monopoles-knots,Correia2009,Dunajski-Hoegner,CDFN83,DS09,DT98,Gunaydin:1995ku,Haupt2011,Haydys,JMPS,Kanno1999-singularities,Kanno-Yasui,Lewis-Thesis,MenetNordstromSaEarp,Munoz-Spin7,Popov2010,PopovSzabo,ReyesCarrion98,SaEarpWalpuski-G2instantons,SaEarp2014-CS,SaEarp-G2overACM,tanaka-spaceDTKahler,Tanaka2014-removal,tanaka-weakcompactness,Tanaka-Spin7,Tian00,Tao-Tian,Walpuski-G2-Kummer,Ward84} (not to mention the long and fruitful study of Hermitian--Yang--Mills connections). 
 In favourable circumstances instantons solve the Yang--Mills equation; for this reason and others the study of instantons informs string theory, supergravity, and theoretical physics.  It is hoped by many that analysing instantons will lead to the construction of invariants of seven-dimensional $G_2$-manifolds, just as was the case for four-dimensional manifolds.

Nearly K\"ahler manifolds were first studied by Wolf and Gray \cite{Gray70,Wolf67,WolfGray1,WolfGray2}.  The lowest dimension in which non-trivial nearly K\"ahler manifolds exist is six, and in this dimension nearly K\"ahler manifolds admit Killing spinors.  Dimension six is also relevant to the theory of special holonomy, as the cone over any nearly K\"ahler six-manifold has holonomy contained in $G_2$ \cite{Bryant87}.  There are precisely four homogeneous nearly K\"ahler six-manifolds \cite{Butruille-vf} (see \cite{Butruille} for an English version).  Until recently these were the only known complete examples, but within the last year new complete examples have been constructed by taking quotients of the homogeneous examples by freely acting discrete groups \cite{CV15} and by analysing the ordinary differential equations that describe nearly K\"ahler metrics with cohomogeneity one \cite{HF15}.

Nearly K\"ahler six-manifolds are a natural arena in which to study instantons.  Instantons on nearly K\"ahler six-manifolds are Yang--Mills \cite[Prop 2.10]{Xu09}, and the tangent bundle over any nearly K\"ahler six-manifold admits an instanton~\cite{HN12}, which is known as the canonical connection and characterised by having skew-symmetric torsion and holonomy contained in $\SU(3)$.

There are two ways in which the study of instantons on nearly K\"ahler six-manifolds informs their study on seven-manifolds with holonomy contained in $G_2$.  The first is through the Bryant--Salamon manifolds \cite{BS89,GPP90}: these are complete non-compact $G_2$-manifolds that asymptote to cones over the homogeneous nearly K\"ahler six-manifolds.  Non-trivial instantons have been constructed on these \cite{Clarke14,Oliveira14} and on the cone over the nearly K\"ahler six-sphere \cite{FN84,FN85}; in all cases the seven-dimensional instanton asymptotes to a non-trivial instanton on the nearly K\"ahler six-manifold.  Thus studying instantons on the Bryant--Salamon manifolds seems to entail studying instantons on nearly K\"ahler six-manifolds.

The second link to $G_2$-geometry is through ``bubbling''.  The instantons on $\RR^7$ constructed in \cite{FN84,FN85} form a one-parameter family.  The parameter describes the size of the instantons and is related to the conformal symmetry of the instanton equations.  At one end of the family the energy density of the instantons spreads out and the instanton converges to a flat connection.  At the other end the energy density becomes concentrated and the instanton converges to a singular connection on $\RR^7\setminus\{0\}$.  The latter is the pull-back of an instanton on $S^6$ (in fact, of the canonical connection).  This example suggests that instantons on $G_2$-manifolds could form point-like singularities whilst maintaining finite energy; such a process would be consistent with the results of Tao and Tian \cite{Tian00,Tao-Tian}.

This paper concerns the deformation theory for instantons on nearly K\"ahler six-manifolds.  We show that the space of solutions to the linearisation of the instanton equations about a given instanton can be identified with a subspace of the kernel of a Dirac operator (in fact, under mild assumptions it is identified with the whole of the kernel).  The Dirac operator has index zero, so one expects instantons to be rigid and their moduli spaces to consist of isolated points.  We confirm this expectation in a number of examples, including those of abelian instantons and of the canonical connection on the six-sphere.  We similarly analyse the allowed perturbations of the canonical connection on the remaining three homogeneous nearly K\"ahler six-manifolds.  In some cases we find non-zero spaces of solutions to the linearised instanton equations, so the construction of new instantons by perturbing known examples remains a tantalising possibility.

In Section \ref{sec:geomNK}, we review the geometry of nearly K\"ahler six-manifolds from a spinorial point of view.  In Section \ref{sec:instdef}, we introduce the deformation theory for instantons on nearly K\"ahler six-manifolds.  In Section \ref{sec:insthomNK} and Section \ref{sec:frobenius}, we apply this theory to investigate perturbations of some homogeneous examples of instantons.  We note that a proof of the rigidity of the canonical connection on $S^6$ was previously claimed in \cite[Thm 3.5]{Xu09}.  The proof given in that paper was unfortunately incorrect, as explained in Section \ref{sec:insthomNK}, but our analysis in Section \ref{sec:frobenius} confirms that this instanton is indeed rigid.
The paper closes with a few appendices in which technical details are provided.

\section{Geometry of nearly K\"ahler manifolds}
\label{sec:geomNK}

Let $M$ be a six-dimensional Riemannian manifold and let $\nabla^{LC}$ denote the Levi-Civita connection on $M$.  The manifold $M$ is called \emph{nearly K\"ahler} if there is a real non-zero constant $\lambda$ and a non-zero section $\psi$ of the real spinor bundle such that
\begin{equation}
\label{KS}
 \nabla^{LC}_X\psi = \lambda X\cdot\psi\quad \forall X\in\Gamma(TM).
\end{equation}
The section $\psi$ is called a \emph{real Killing spinor}.  Any nearly K\"ahler manifold admits at least two Killing spinors, since the section $\Vol\cdot\psi$ satisfies
\[ \nabla^{LC}_X \Vol\cdot\psi = - \lambda X\cdot\Vol\cdot\psi\quad \forall X\in\Gamma(TM). \]
By rescaling the metric and if necessary replacing $\psi$ with $\Vol\cdot\psi$ the sign and magnitude of $\lambda$ can be altered to any given value; therefore for simplicity this document uses a convention in which
\[ \lambda = \frac12. \]

Any six-dimensional nearly K\"ahler manifold is automatically Einstein \cite{BFGK91}, with Ricci curvature $\Ric = 5 g$.  Moreover, a six-dimensional nearly K\"ahler manifold admits an SU(3) structure, essentially because the stabiliser of any real spinor in six dimensions is isomorphic to SU(3).  The SU(3)-structure is characterised by an almost complex structure, a K\"ahler form, and a holomorphic volume form.  All of these may be constructed directly from the Killing spinor, as explained below.

Let $(V,g)$ denote a real six-dimensional vector space $V$ equipped with a positive definite metric $g$.  Recall that $\Cl(V,g)\cong \Lambda^*V$ as vector spaces.   Many algebraic expressions are very easy to prove.  We record a few here in a lemma for they are used often in the course of this paper.

\begin{lemma} 
	\label{lemma:algebra}
Let $\alpha\in \Lambda^1V$ and $\beta\in \Lambda^pV$.  Then in $\Cl(V,g)$,
\begin{align}
	[\alpha,\beta]&=\begin{cases} 2\alpha\wedge \beta,&\text{ if $p$ is odd,}\\
								-2\alpha\lrcorner \beta,&\text{ if $p$ is even,}
					\end{cases}\\
	\{\alpha,\beta\}&=\begin{cases} -2\alpha\lrcorner \beta,&\text{ if $p$ is odd,}\\
								2\alpha\wedge \beta,&\text{ if $p$ is even.}
					\end{cases}
\end{align}
\end{lemma} 

Recall that the space $S$ of spinors for $(V,g)$ is a real eight-dimensional vector space equipped with a positive-definite symmetric bilinear form $(\cdot,\cdot)$ which carries a representation of the Clifford algebra $\Cl(V,g)$.  Let $\psi\in S$ be any spinor of unit length, and let $\psi^T\in S^\ast$ be the conjugate of $\psi$ with respect to the symmetric bilinear form.  Then $\psi\otimes\psi^T$ is a self-adjoint element of $S\otimes S^\ast\cong \Cl(V,g)$.

The self-adjoint subspace of $\Cl(V,g)$ is identified under the canonical isomorphism $\Cl(V,g)\cong \Lambda^\ast V$ with the subspace $\Lambda^0\oplus \Lambda^3\oplus\Lambda^4$.  Therefore $\psi\otimes\psi^T$ defines unique forms of degrees zero, three, and four.  The zero-form is equal to $\sfrac18$, because $\sfrac18\Tr_S (\psi\otimes\psi^T) = \sfrac18$.  Thus $\psi$ uniquely determines $P\in\Lambda^3 V$ and $Q\in\Lambda^4V$ through the equation
\begin{equation}
\label{def PQ}
8\psi\otimes\psi^T = 1 + P - Q.
\end{equation}
Let us note here for future reference that the stabiliser subgroup in $\mathrm{Spin}(V,g)$ of $\psi$ is isomorphic to SU(3), and that the corresponding subgroup of $\mathrm{SO}(V,g)$ is again isomorphic to SU(3).

The spinor $\psi$ defines a linear map $\phi\mapsto\phi\cdot\psi$ from $\Lambda^0\oplus\Lambda^1\oplus\Lambda^6$ to $S$.  This map is easily seen to be an isometry with respect to the metrics $g$ and $(\cdot,\cdot)$ so must be injective.  Since its domain and target have equal dimension it is an isomorphism:
\begin{equation}\label{spinor isomorphism} S \cong \Lambda^0 V \oplus \Lambda^1 V\oplus\Lambda^6 V. \end{equation}

\begin{lemma} \label{evals PQ}
 The subspaces of $S$ isomorphic to $\Lambda^0$, $\Lambda^6$ and $\Lambda^1$ are eigenspaces of the operations of Clifford multiplication by $P$ and $Q$.  The associated eigenvalues are 
 \begin{center}\begin{tabular}{c|ccc}
  & $\Lambda^0$ & $\Lambda^1$ & $\Lambda^6$ \\
  \hline
  P & 4 & 0& -4  \\
  Q & -3 &1& -3. 
 \end{tabular}\end{center}
\end{lemma}

\begin{proof}
 The operations of multiplication by $P$ and $Q$ are SU(3)-equivariant, since $P$ and $Q$ are constructed from $\psi$ and $\psi$ is SU(3)-invariant.  Therefore the subspaces of $S$ isomorphic to $\Lambda^0\oplus \Lambda^6$ and $\Lambda^1$ are fixed by $P$ and $Q$.  The subspace $\Lambda^1$ forms an irreducible representation of SU(3), so $P$ and $Q$ act by scalar multiplication on this subspace.

 Since the action of $Q$ is self-adjoint and commutes with the action of $\Vol$ on $\Lambda^0\oplus\Lambda^6$, $Q$ must act as multiplication by some real constant $q_1$ on this space.  Since the action of $P$ is self-adjoint and anti-commutes with the action of $\Vol$, there must exist constants $p_1$ and $p_2$ such that $P\cdot\psi = p_1\psi+p_2\Vol\cdot\psi$ and $P\cdot\Vol\cdot\psi = p_2\psi-p_1\Vol\cdot\psi$.  Given additionally that the actions of $P$ and $Q$ are both traceless, they must take the following form with respect to the decomposition $S\cong \Lambda^0\oplus\Lambda^6\oplus\Lambda^1$:
 \[
  P = \left(\begin{array}{cc|c} p_1 & p_2 & 0 \\ p_2 & -p_1 & 0 \\ \hline 0 & 0 & 0 \end{array}\right),\quad
  Q = \left(\begin{array}{cc|c} q_1 & 0 & 0 \\ 0 & q_1 & 0 \\ \hline 0 & 0 & -q_1/3 \end{array}\right).
 \]
 Equation \eqref{def PQ} then implies that $1+p_1-q_1=8$, $1-p_1-q_1=0$, $p_2=0$ and $1+q_1/3=0$.  The unique solution of this system of equations is $p_1=4$, $p_2=0$ and $q_1=-3$, giving the advertised result.
\end{proof}

A complex structure may be defined on $V$ using the isomorphism given in Equation~\eqref{spinor isomorphism}.  If $u\in V$ then $\Vol\cdot u\cdot \psi$ belongs to the subspace $\Lambda^1 V\subset S$; therefore we may define $Ju$ through the equation
\[ J u\cdot \psi = \Vol\cdot u\cdot\psi. \]

Having identified a complex structure one may define a K\"ahler form $\omega$ in the usual way and a unique (up to normalisation) complex 3-form $\Omega$ of type (3,0).
Although these are not needed in what follows we pause to explain how these are related to the forms $P$ and $Q$.  With suitable normalisation of $\Omega$, it holds that
\begin{equation}\label{omegas} \omega = \ast Q\qquad\mbox{and that}\qquad \Omega=P+\ii\ast P. \end{equation}
Proofs of these equations are supplied in Appendix \ref{sec:omegaOmega}.  Let us remark here that Lemma \eqref{evals PQ} implies that $\|P\|^2 = \sfrac18\Tr_S(P^2) = 4$.  This relation implies that $\Omega$ is given a standard normalisation in which
\[ \Omega = (e^1+\ii e^2)\wedge(e^3+\ii e^4)\wedge(e^5+\ii e^6) \]
in some orthonormal basis $e^1,\ldots,e^6$.

If $\psi$ solves the real Killing spinor Equation~\eqref{KS} then its length is constant.  Therefore any six-dimensional nearly K\"ahler manifold admits non-vanishing forms $P$ and $Q$ defined as above.  Since $\psi$ is non-vanishing it defines an SU(3)-structure on $M$, such that $\psi$ and the forms $P$ and $Q$ are parallel with respect to any connection with holonomy group contained in SU(3).

\begin{lemma}\label{derivatives PQ}
 The differential forms $P$ and $Q$ satisfy the differential identities,
 \[ \dd P = 4 Q,\quad \dd \ast Q=3\ast P. \]
\end{lemma}

\begin{proof}
 The exterior derivative of any form $\phi$ may be calculated using the Levi-Civita connection via the identity
 \[ \dd\phi = \nabla^{LC}\wedge\phi. \]
 Let $e^a$ be a local orthonormal frame for the cotangent bundle, and write $\nabla^{LC} = e^a\otimes \nabla_a^{LC}$.  Then, from Equation~\eqref{def PQ} defining $P$ and $Q$ and the Killing spinor equation Equation~\eqref{KS},
 \begin{align*}
  \nabla\wedge (1+P-Q) &= \sum_{a=1}^6 e^a\wedge \frac{1}{2}[e^a,1+P-Q] \\
                       &= \sum_{a=1}^6 e^a\wedge (e^a\wedge P + e^a\lrcorner Q) \\
                       &= 4 Q.
 \end{align*}
 Therefore $\dd P= 4Q$ and $\dd Q=0$.  Similarly, since $P\cdot\Vol = \ast P$ and $Q\cdot\Vol = \ast Q$,
 \begin{align*}
  \nabla\wedge (\Vol+\ast P-\ast Q) &= \sum_{a=1}^6 e^a\wedge \frac{1}{2}[e^a,(1+P-Q)]\cdot\Vol \\
                                    &= \sum_{a=1}^6 e^a\wedge \frac{1}{2}\{e^a,\Vol+\ast P-\ast Q\} \\
                                    &= \sum_{a=1}^6 e^a\wedge (-e^a\lrcorner\ast P-e^a\wedge \ast Q\} \\
                                    &= -3\ast P.
 \end{align*}
 Therefore $\dd\ast Q=3\ast P$ and $\dd\ast P = 0$.
\end{proof}
It should also be noted that the real Killing spinor equation is equivalent to the differential equations
\[ (\nabla_X J)X=0\quad\forall X\in\Gamma(TM), \]
for $J$ \cite{Grunewald90}.  Thus K\"ahler six-manifolds may equivalently be defined to be almost Hermitian manifolds whose non-integrable almost complex structure satisfies this identity.

Vector fields $X$ that preserve the metric $g$ and the Killing spinor $\psi$ are called \emph{automorphic}.   If $X$ is automorphic, then $X$ preserves also $P,Q,\omega,\Omega$ and $J$.

A key feature of nearly K\"ahler geometry is the presence of a distinguished connection on the tangent bundle with holonomy SU(3) and skew parallel torsion.  Let $t\in\RR$ be a parameter, and let $\nabla^t$ be the connection constructed from the Levi-Civita connection $\nabla^{LC}$ as follows:
\begin{equation}
 g(\nabla^t_X Y,Z ) = g(\nabla^{LC}_X Y, Z ) + \frac{t}{2}P(X,Y,Z)\quad \forall X,Y,Z\in \Gamma(TM).
\end{equation}
The torsion tensor $T^t$ of the connection $\nabla^t$ is proportional to $P$:
\[
 g(X,T^t(Y,Z)) = t P(X,Y,Z).
\]
The connection $\nabla^t$ acts on sections $\eta$ of the spin bundle as follows:
\begin{equation}
\label{spin connection}
\nabla^t_X \eta = \nabla^{LC}_X \eta + \frac{t}{4} (i_X P)\cdot\eta.
\end{equation}
It follows from Equation~\eqref{KS} and Lemma \ref{evals PQ} that, for any vector field $X$,
\begin{align*}
\nabla^t_X\psi &= \frac{1}{2} X\cdot\psi - \frac{t}{8} (X\cdot P + P\cdot X)\cdot\psi \\
&= \frac{1-t}{2} X\cdot\psi.
\end{align*}
Therefore $\psi$ is parallel with respect to the connection $\nabla^1$, and $\nabla^1$ has holonomy group contained in SU(3).  The connection $\nabla^1$ is known as the ``canonical'' or ``characteristic'' connection.

For later use, we note here the following formula for the Ricci curvature tensor of the connection $\nabla^t$:
\begin{proposition}\label{prop:Ricci}
 The Ricci tensor $\Ric^t$ of the connection $\nabla^t$ is equal to $(5-t^2)$ times the metric $g$.
\end{proposition}
\begin{proof}
 Friedrich and Ivanov derive \cite{Friedrich-Ivanov-skewtorsion} an expression for the Ricci tensor of a connection with totally skew-symmetric torsion in terms of the Ricci tensor of the Levi-Civita connection and the torsion 3-form.  In the case where the torsion three-form is $tP$ their formula reads
 \[ \Ric^0(X,Y) = \Ric^t(X,Y) + \frac{t}{2} \dd^\ast P (X,Y) + \frac{t^2}{2} g(i_XP, i_YP). \]
 By Lemma \ref{derivatives PQ}, $\dd^\ast P= 0$.  As has already been noted, the Ricci tensor $\Ric^0$ of the Levi-Civita connection equals 5 times the metric $g$.

 The remaining term may be evaluated by a direct calculation, using the identity
 \[ g(i_XP,i_YP) = -\frac{1}{32}\Tr_S(\{X,P\}\{Y,P\}). \]
 With the aid of Lemma \ref{evals PQ} one finds that
 \begin{align*}
  \{X,P\} \psi &= X\cdot P\cdot\psi \\
               &= 4X\cdot\psi,
 \end{align*}
 and similarly that
 \begin{align*}
  \{X,P\} \Vol\cdot \psi &= -4X\cdot\Vol\cdot\psi \\
               &= 4JX\cdot\psi.
 \end{align*}
 Finally, $\{X,P\}\cdot Z\cdot\psi$ can be evaluated by taking inner products with elements of $SM$.  It holds that
 \begin{align*}
 (\psi,\{X,P\}\cdot Z\cdot\psi) &= (\psi,P\cdot X\cdot Z\cdot\psi) \\
 &= (P\cdot\psi,X\cdot Z\cdot\psi) \\
 &= 2\Bigl(\psi,\bigl([X,Z]-2g(X,Z)\bigr)\psi\Bigr) \\
 &= -4g(X,Z)
 \end{align*}
 since the element $[X,Z]$ in the Clifford algebra is an antisymmetric endomorphism of $SM$.
 Similarly, $(\Vol\cdot\psi,\{X,P\}\cdot Z\cdot\psi)=-4g(JX,Z)\psi$ and $(W\cdot\psi,\{X,P\}\cdot Z\cdot\psi)=0$ for all $W\in TM$.  Therefore
 \begin{align*}
  \{X,P\} Z\cdot\psi = -4g(X,Y)\psi -4g(JX,Z)\Vol\cdot\psi.
 \end{align*}
 These formulae together allow evaluation of the trace:
 \[ -\frac{1}{32}\Tr_S(\{X,P\}\{Y,P\}) = 2g(X,Y). \]
 The result follows.
\end{proof}

\section{Instantons and deformations}
\label{sec:instdef}

Let $A$ be a connection on a principal $K$-bundle $\PPP$ over a nearly K\"ahler six-manifold $(M,g,\psi)$ and let $F$ be its curvature.  Then $A$ is called an \emph{instanton} if its curvature $F$ satisfies
\begin{equation}
\label{instanton eq 1}
 F\cdot\psi = 0.
\end{equation}
Note that in this equation only the two-form part of $F$ is acting on $\psi$.  Thus if the adjoint bundle associated to the principal bundle $\PPP$ is denoted $\Ad_\PPP$, the left hand side is a section of $\Ad_\PPP\otimes SM$.

The instanton Equation~\eqref{instanton eq 1} can be reformulated in a number of ways.  Firstly, the SU(3)-structure defines a subbundle $\su(3)M$ of $\End(TM)$ and also, via the metric-induced isomorphism $\End(TM)\cong \Lambda^2M$, of $\Lambda^2 M$.   The instanton Equation~\eqref{instanton eq 1} is equivalent to the statement,
\begin{equation}
 \label{instanton eq 2}
 F\in \Gamma(\su(3)M\otimes \Ad_\PPP)\subset\Gamma(\Lambda^2M\otimes \Ad_\PPP).
\end{equation}

Secondly, as on K\"ahler manifolds, the instanton equation given by Equation~\eqref{instanton eq 2} is equivalent to the Hermitian--Yang--Mills equation
\begin{equation}
\label{instanton eq 4}
F^{2,0}=0,\quad \omega\lrcorner F=0.
\end{equation}

Thirdly, the instanton equation is equivalent to
\begin{equation}
 \label{instanton eq 3}
 F\lrcorner Q = - F.
\end{equation}
It is straightforward to see that Equation~\eqref{instanton eq 1} implies Equation~\eqref{instanton eq 3}: Equation~\eqref{instanton eq 1} implies that the two-form part of $F$ acts trivially on $\Lambda^0\oplus\Lambda^6$ and sends $\Lambda^1$ to itself with respect to the decomposition $S\cong\Lambda^0\oplus\Lambda^1\oplus\Lambda^6$; Lemma \ref{evals PQ} then implies that $F\lrcorner Q = -\frac12\{F,Q\} = -\frac12\{F,1\} = -F$.  It can further be proved that \eqref{instanton eq 3} implies \eqref{instanton eq 4} by classifying the eigenvalues and eigenspaces of the operator on two-forms given by contraction with $Q$ \cite{HILP10}.

Instantons on nearly K\"ahler manifolds have the desirable property of solving the Yang--Mills equation -- see \cite{HN12} for a proof based on Equation~\eqref{instanton eq 1} and the Killing spinor equation, or \cite[Prop 2.10]{Xu09} for a proof based on Equation~\eqref{instanton eq 3} and the differential identities given in Lemma \ref{derivatives PQ}.  The canonical connection $\nabla^1$ on the tangent bundle of a nearly K\"ahler manifold is always an instanton \cite{HN12}.


The purpose of this note is to study perturbations of instantons.  A perturbation of a connection $A$ is a section $\epsilon$ of $\Ad_\PPP\otimes T^\ast M$, and to leading order the corresponding perturbation of the curvature $F$ is $\dd^A\epsilon$.  The gauge freedom in the perturbation can be fixed by imposing the standard condition $\dd^A\ast\epsilon=0$; thus an infinitesimal perturbation of an instanton $A$ is given by a solution $\epsilon$ to the equations,
\begin{equation}
\label{instanton perturbations}
 \dd^A\epsilon\cdot\psi = 0,\quad \dd^A\ast\epsilon = 0.
\end{equation}
The purpose of the next proposition is to identify solutions of the above system with eigenmodes of a Dirac operator.

\begin{proposition}
 Let $\epsilon$ be a section of $\Ad_\PPP\otimes T^\ast M$, let $t\in\RR$, and let $D^{t,A}$ be the Dirac operator constructed from the connections $\nabla^t$ and $A$.  Then $\epsilon$ solves equations \eqref{instanton perturbations} if and only if
 \begin{equation}\label{spinor instanton perturbations} D^{t,A} (\epsilon\cdot\psi) = 2 \epsilon\cdot\psi. \end{equation}
\end{proposition}
\begin{proof}
 Let $e^a$ be a local orthonormal frame for $T^\ast M$.  The identity
 \[
  D^{0,A} (\epsilon\cdot\psi) = (\dd^A \epsilon + (\dd^A)^\ast\epsilon)\cdot\psi + e^a\cdot\epsilon\cdot \nabla^0_a\psi
 \]  
is easily verified. The Killing spinor Equation~\eqref{KS} and the identity $e^a\cdot \epsilon\cdot e^a = 4\epsilon$ then imply that
 \[
  D^{0,A}(\epsilon\cdot\psi) = (\dd^A \epsilon + (\dd^A)^\ast\epsilon+2\epsilon)\cdot\psi.
 \]

 It follows from Equation~\eqref{spin connection} that the Dirac operator $D^{t,A}$ is given by
 \[ D^{t,A} = D^{0,A} + \frac{3t}{4} P. \]
 Then by Lemma \ref{evals PQ}, we have
 \[
  D^{t,A}(\epsilon\cdot\psi) = (\dd^A \epsilon + (\dd^A)^\ast\epsilon+2\epsilon)\cdot\psi.
 \]

 From this identity one obtains that the equation $D^{t,A} (\epsilon\cdot\psi) = 2 \epsilon\cdot\psi$ is equivalent to $(\dd^A \epsilon + (\dd^A)^\ast\epsilon)\cdot\psi=0$.  The latter is equivalent to the pair \eqref{instanton perturbations} of equations, because the two components $\dd^A \epsilon\cdot\psi$ and $(\dd^A)^\ast\epsilon\cdot\psi$ belong to the complementary subspaces $(\Lambda^1M\oplus\Lambda^6M)\otimes\Ad_\PPP$ and $\Lambda^0M\otimes\Ad_\PPP$    of $SM\otimes\Ad_\PPP$.
\end{proof}

It is worth noting that Equation~\eqref{spinor instanton perturbations} is independent of $t$.  
The linearised instanton equations can also be formulated as part of an elliptic complex \cite{ReyesCarrion98}.  
However, we have found the spinorial formulation more practical to work with, not least because there already exists a large body of literature containing useful identities for Dirac operators with torsionful connections.

By the previous proposition, to prove that an instanton is rigid it suffices to prove that 2 does not belong to the spectrum of the restriction of the Dirac operator $D^{t,A}$ to $\Lambda^1M\otimes\Ad_\PPP\subset SM\otimes\Ad_\PPP$.  To this end, we need a Schr\"odinger--Lichnerowicz formula for the square of the Dirac operator.  Such a formula has been obtained in the case $A=0$ by Agricola and Friedrich \cite{AF04}; the following proposition provides a gauged version of their formula.  A complete proof is presented below in order to keep this discussion self-contained.
\begin{proposition}\label{torsionful SL formula}
 Let $EM$ be the vector bundle obtained from $P$ through some representation $E$ of $G$ and let $\eta\in\Gamma(EM\otimes SM)$.  Let $A$ be any connection on $P$ and let $t\in\RR$.  Then
\begin{equation}
\label{SLNK0}
 (D^{t/3,A})^2\eta = (\nabla^{t,A})^\ast\nabla^{t,A}\eta + \frac14\Scal_g\eta + \frac{t}{4}\dd P\cdot\eta - \frac{t^2}{8}\|P\|^2\eta + F\cdot\eta.
\end{equation}
 (Note that in this formula the two-form part of $F$ acts by Clifford multiplication on $SM$ and the $\Ad_\PPP$ part of $F$ acts on $EM$ in the usual way)
\end{proposition}
\begin{proof}
 Let $e^1,\ldots,e^6$ be a local orthonormal frame for the tangent bundle.  The square of the Dirac operator expands as follows:
 \begin{align*}
 (D^{t/3,A})^2\eta &= \left(D^{0,A} + \frac{t}{4}P\right)^2\eta \\
                   &= (D^{0,A})^2\eta + \frac{t}{4}\{D^{0,A},P\}\eta + \frac{t^2}{16} P\cdot P\cdot\eta.
 \end{align*}
 By the usual Schr\"odinger--Lichnerowicz formula, the first of the three terms on the right of this expression is
 \[ (D^{0,A})^2\eta = (\nabla^{0,A})^*\nabla^{0,A}\eta + \frac14\Scal_g\eta + F\cdot\eta. \]
 The second term is
 \begin{align*}
  \frac{t}{4}\{D^{0,A},P\}\eta &= -\frac{t}{2}(e^a\lrcorner P)\cdot \nabla^{0,A}_a\eta + \frac{t}{4}\dd P\cdot\eta + \frac{t}{4}\dd^\ast P\cdot\eta.
  \end{align*}
  One simplifies the third term using the identity $\alpha\cdot\alpha=\|\alpha\|^2-(e^a\lrcorner \alpha\wedge e^a\lrcorner\alpha)$ valid for any three-form $\alpha$.  Note that the expression for $\alpha\cdot \alpha$ has no components in $\Lambda^2M$ or $\Lambda^6M$, because $\alpha\cdot \alpha$ is self-adjoint while two-forms and six-forms are skew-adjoint.   Thus
  \begin{align*}
  \frac{t^2}{16} P\cdot P\cdot\eta &= \frac{t^2}{16} \|P\|^2\eta - \frac{t^2}{16} (e^a\lrcorner P\wedge e^a\lrcorner P)\cdot\eta.
 \end{align*}

Now, for any two-form  $\beta$, we have $\beta\cdot\beta=-\|\beta\|^2+\beta\wedge\beta$.  Given that $\sum_a\|e^a\lrcorner P\|^2=3\|P\|^2$, we have $(e^a\lrcorner P)\cdot (e^a\lrcorner P)=-3\|P\|^2+(e^a\lrcorner P)\wedge (e^a\lrcorner P)$.  Hence at the center of a normal frame (where the Christoffel symbols vanish), we have
 \begin{align*}
  (\nabla^{t,A})^\ast\nabla^{t,A}\eta &= -\left(\nabla^{0,A}_a+\frac{t}{4}e^a\lrcorner P\right)\left(\nabla^{0,A}_a+\frac{t}{4}e^a\lrcorner P\right)\eta \\
                                      &= -\nabla^{0,A}_a\nabla^{0,A}_a\eta - \frac{t}{2}(e^a\lrcorner P)\cdot \nabla^{0,A}_a\eta + \frac{t}{4}\dd^\ast P\cdot\eta \\
                                      & \qquad\qquad\qquad+ \frac{t^2}{16} (3\|P\|^2-e^a\lrcorner P\wedge e^a\lrcorner P)\cdot\eta.
 \end{align*}

 Combining the above equations yields the desired result.
\end{proof}

The proof of the preceding proposition is very general and makes no assumptions about the dimension of $M$, the existence of Killing spinors, or whether $A$ is an instanton.  
In the case of interest,  the scalar curvature is equal to 30 (when $\lambda=1/2$), $\|P\|^2=4$, and $\dd P =4Q$ (see Lemma \ref{derivatives PQ}), hence
\begin{equation}
    (D^{t/3,A})^2 \eta = (\nabla^{t,A})^\ast\nabla^{t,A}\eta + \left(\frac{15-t^2}{2}\right)\eta  +tQ\cdot \eta+F\cdot \eta.
\end{equation}
This formula should be compared with \cite[Equation (2)]{AFK08} in the case $A=0$, $t=1$.

From Lemma \ref{evals PQ}  one learns that $Q$ acts as multiplication by 4 on $\Lambda^1 M\subset SM$ and as multiplication by $-3$ on $\Lambda^0 M\oplus \Lambda^6M\subset SM$.   By virtue of the instanton equation, the curvature term acts trivially on $\Lambda^0 M\subset SM$ while
\[ F\cdot \epsilon\cdot\psi = - 2 (F\llcorner \epsilon)\cdot\psi, \]
where it should be noted that $F$ acts on $\epsilon$ by contraction of forms and via the action of the Lie algebra of the gauge group on $E$.  Hence we obtain for $\eta\in\Gamma\Bigl(EM\otimes \bigl(\Lambda^0 M\oplus \Lambda^6M\bigr)\Bigr)$
\begin{align}
\label{SLNK16}	(D^{t/3,A})^2 \eta &= (\nabla^{t,A})^\ast\nabla^{t,A}\eta + \left(\frac{15}{2} -3t-\frac{t^2}{2} \right)\eta,\text{ while}\\
 \label{SLNK}
 (D^{t/3,A})^2 (\epsilon\cdot\psi) &= (\nabla^{t,A})^\ast\nabla^{t,A}(\epsilon\cdot\psi) + \left(\frac{15}{2} + t-\frac{t^2}{2} \right)\epsilon\cdot\psi - 2 (F\llcorner \epsilon)\cdot\psi.
\end{align}

The most useful case of the identity \eqref{SLNK} is when $t=1$, for the following two reasons.  Firstly, $t=1$ is the value that maximises the right hand side of the identity, and hence yields the strongest lower bound on the square of the Dirac operator.  Secondly, when $t=1$ the Laplace operator on the right hand side of the identity is the one for the canonical connection, which (as demonstrated in the next section) has useful representation-theoretical properties on homogeneous spaces.  Since $\psi$ is parallel with respect to $\nabla^1$, the $t=1$ case of the identity is equivalent to
\begin{equation}
 \label{SLNK2}
 (D^{1/3,A})^2 (\epsilon\cdot\psi) = \left((\nabla^{1,A})^\ast\nabla^{1,A}\epsilon + 8\epsilon - 2 F\llcorner \epsilon\right)\cdot\psi.
\end{equation}

From the point of view of analysing instantons, the most useful case of Proposition \ref{torsionful SL formula} is when the vector bundle $EM$ equals $\Ad_\PPP$.  Let $H\subset K$ denote the holonomy group of the connection $A$.  The group $H$ acts on the Lie algebra $\kk$ of $K$.  Let $\kk_1\subset\kk$ be the subspace on which  $H$ acts trivially, and suppose that there is a complementary subspace $\kk_0\subset\kk$ such that $\kk\cong\kk_0\oplus\kk_1$ (when $\kk$ admits an $H$-invariant non-degenerate bilinear form, such a complementary subspace exists).  There is a corresponding splitting of the adjoint bundle:
\[ \Ad_\PPP = L_0\oplus L_1. \]
\begin{proposition}\label{complexification}
Let $A$ be an instanton on $\PPP$ with holonomy group $H$ and suppose that $\Ad_\PPP$ splits as above.  Then
\begin{enumerate}[(i)]
\item $\ker((D^{1/3,A})^2-4) = \ker((D^{1/3,A})^2-4) \cap (\Omega^1L_0\oplus\Omega^0L_1\oplus \Omega^6L_1$);
\item $\ker((D^{1/3,A})^2-4) \cap (\Omega^0L_1\oplus\Omega^6L_1) \cong 2\kk_1$;
\item $\ker((D^{1/3,A})^2-4) \cap \Omega^1L_0 \cong 2\Bigl(\ker(D^{1/3,A}-2) \cap \Omega^1L_0\Bigr)$.
\end{enumerate}
Moreover, the volume form induces on both $\ker((D^{1/3,A})^2-4) \cap (\Omega^0L_0\oplus\Omega^6L_0)$ and $\ker((D^{1/3,A})^2-4) \cap \Omega^1L_0$ almost-complex structures that swap the two copies of the vector space in the decompositions given above.
\end{proposition}
\begin{proof}
Note that $\nabla^{1,A}$ respects the decomposition
\begin{multline*}
SM\otimes \Ad_\PPP = (\Lambda^1M\otimes L_0)\oplus(\Lambda^1M\otimes L_1)\\
\oplus((\Lambda^0M\oplus\Lambda^6M)\otimes L_0)\oplus ((\Lambda^0M\oplus\Lambda^6M)\otimes L_1),
\end{multline*}
as does $(D^{1/3,A})^2$ (by Proposition \ref{torsionful SL formula}).  Therefore to establish the first identity we must show that $\ker((D^{1/3,A})^2-4)\cap \Omega^1L_1$ and $\ker((D^{1/3,A})^2-4)\cap (\Omega^0L_0\oplus\Omega^6L_0)$ are zero-dimensional vector spaces.  If $\epsilon\cdot\psi\in\ker((D^{1/3,A})^2-4)\cap \Omega^1L_1$ then the term involving $F$ in Equation~\eqref{SLNK2} is zero (because $H$ acts trivially on $\kk_1$).  So
\begin{align*}
0 
&= \int_M (\epsilon\cdot\psi, ((D^{1/3,A})^2-4)\epsilon\cdot\psi) \Vol \\
&= \int_M\left[ (\nabla^{1,A}\epsilon\cdot\psi, \nabla^{1,A}\epsilon\cdot\psi)+4(\epsilon\cdot\psi,\epsilon\cdot\psi)\right]\Vol \\
&\geq 4\int_M(\epsilon\cdot\psi,\epsilon\cdot\psi)\Vol.
\end{align*}
Therefore $\epsilon\cdot\psi=0$.  If $\eta\in \ker((D^{1/3,A})^2-4)\cap (\Omega^0L_0\oplus\Omega^6L_0)$, we use Equation~\eqref{SLNK16} to get
\begin{align*}
0 
&= \int_M (\eta, ((D^{1/3,A})^2-4)\eta) \Vol \\
&= \int_M(\nabla^{1,A}\eta, \nabla^{1,A}\eta)\Vol \\
&\geq 0.
\end{align*}
Therefore $\nabla^{1,A}\eta=0$.  However, $\Lambda^0L_0\oplus\Lambda^6L_0$ has no non-zero parallel sections, because the fibre $\kk_0$ of $L_0$ has no non-zero elements fixed by the action of the holonomy group $H$ of $A$ (and the holonomy group of $\nabla^1$ acting on $\Lambda^0\oplus\Lambda^6$ is trivial).

To establish the second identity we argue as above that any element of $\ker((D^{1/3,A})^2-4)\cap (\Omega^0L_1\oplus\Omega^6L_1)$ is parallel.  By the general holonomy principle the space of parallel sections of $\Lambda^0L_1$ is isomorphic to the fixed set of $H$ in $\kk_1$, which is the whole of $\kk_1$ by definition.  Similarly, the space of parallel sections of $\Lambda^6L_1$ is isomorphic to $\kk_1$.  The sum of these two spaces is naturally isomorphic to $\kk\otimes\CC$, with the almost complex structure given by multiplication with $\Vol$.

To establish the third identity we note that the connection $A$ fixes the subbundle $L_0\subset \Ad_\PPP$ and, by the first part of this proposition, $\ker((D^{1/3,A})^2-4) \cap \Gamma(SM\otimes L_0)=\ker((D^{1/3,A})^2-4) \cap \Omega^1L_0$ hence the space $\ker((D^{1/3,A})^2-4) \cap \Omega^1L_0$ is mapped into itself by the operator $D^{1/3,A}$.   Therefore
\[ \ker((D^{1/3,A})^2-4) \cap \Omega^1L_0 = \ker(D^{1/3,A}-2) \cap \Omega^1L_0\oplus \ker(D^{1/3,A}+2) \cap \Omega^1L_0. \]
Multiplication by the volume form defines a linear map from this vector space to itself.  This map squares to $-1$ so is an almost complex structure.  It also swaps the two summands on the right hand side because it anti-commutes with $D^{1/3,A}$.  Therefore the total space is isomorphic to the complexification of one of the two factors.
\end{proof}

The preceding proposition has important consequences in two particular cases.  Firstly, if the structure group $K$ is abelian then $\kk_0=0$ and $\ker((D^{1/3,A})^2-4)\cap\Omega^1 \Ad_\PPP$ is zero-dimensional, so the space of deformations of the instanton is a subspace of a zero-dimensional space and hence is zero-dimensional.  Thus:
\begin{theorem}\label{abelian rigid}
Any instanton on a principal bundle with abelian structure group is rigid.
\end{theorem}
Secondly, we recall that a connection on a principal bundle is called \emph{irreducible} if its holonomy group equals the structure group of the principal bundle.  If $A$ is an irreducible connection and the structure group of $\PPP$ is semisimple then $\kk_1=0$.  In this case the previous proposition implies that:
\begin{theorem}
The space of deformations of an irreducible instanton on a principal bundle with semisimple structure group is isomorphic to the kernel of the elliptic operator
\[ (D^{1/3,A}-2): \Gamma(SM\otimes \Ad_\PPP)\to\Gamma(SM\otimes \Ad_\PPP). \]
\end{theorem}

We end this section with some comments on a geometrical interpretation of the eigenspace of $(D^{1/3,A})^2$ acting on $\Omega^1\Ad_\PPP$ with eigenvalue 4.  Given any orthogonal connection $\nabla$ on the tangent bundle and any connection $A$ on a principal bundle with curvature $F$, the \emph{$\nabla$-Yang--Mills equation} for $A$ is
 \[ (\nabla^A)^\ast F = 0. \]
When $\nabla$ is the Levi-Civita connection, this equation is just the usual Yang--Mills equation.  The instanton equation on a nearly K\"ahler six-manifold implies the $\nabla^t$-Yang--Mills equation for any $t\in\RR$.  Indeed, the term involving the torsion is proportional to $F\lrcorner P$, and  vanishes as $F$ is a (1,1)-form while $P$ is the real part of a (3,0)-form.

The Yang--Mills equation for the canonical connection $\nabla^1$,
 \[ (\dd^A)^\ast F + F\lrcorner P = 0, \]
linearises to
\begin{equation}\label{TYM} (\dd^A)^\ast\dd^A\epsilon - F\llcorner\epsilon + \dd^A\epsilon\lrcorner P = 0, \end{equation}
with $\epsilon$ a section of $\Ad_\PPP\otimes T^\ast M$.  The next proposition proves an identity which relates this equation to the operator $(D^{1/3,A})^2$.
\begin{proposition}
 Let $EM$ be the vector bundle obtained from a $K$-principal bundle $\PPP$ over a nearly K\"ahler six-manifold through some representation $E$ of $K$ and let $\epsilon\in\Gamma(EM\otimes T^\ast M)$.  Let $A$ be any connection on $\PPP$.  Then
 \[ (D^{1/3,A})^2(\epsilon\cdot\psi)-4\epsilon\cdot\psi = \left( \dd^A(\dd^A)^\ast\epsilon + (\dd^A)^\ast\dd^A\epsilon - F\llcorner\epsilon + \dd^A\epsilon\lrcorner P\right)\cdot\psi. \]
\end{proposition}
\begin{proof}
 The Weitzenb\"ock identity states that
 \[ \left( (\dd^A)^\ast\dd^A + \dd^A(\dd^A)^\ast\right) \epsilon = (\nabla^{0,A})^\ast\nabla^{0,A}\epsilon + \Ric^0\epsilon - F\llcorner\epsilon. \]
 With our normalisation conventions, $\Ric^0$ is equal to 5 times the identity.  The Laplacian $(\nabla^{0,A})^\ast\nabla^{0,A}$ is related to the Laplacian $(\nabla^{1,A})^\ast\nabla^{1,A}$ that appears in Equation~\eqref{SLNK2} as follows:
 \[ (\nabla^{1,A})^\ast\nabla^{1,A}\epsilon = (\nabla^{0,A})^\ast\nabla^{0,A}\epsilon + \dd^A\epsilon\lrcorner P+\epsilon. \]
 Combining the above two identities with Equation~\eqref{SLNK2} yields the desired result.
\end{proof}

Suppose now that $A$ is an instanton and $\epsilon$ is a section of $\Lambda^1M\otimes\Ad_\PPP$ which solves the linearised torsionful Yang--Mills Equation~\eqref{TYM} and the gauge-fixing condition $(\dd^A)^\ast\epsilon=0$.  Then by the previous proposition $(D^{1/3,A})^2\epsilon\cdot\psi=4\epsilon\cdot\psi$.  Conversely, if $(D^{1/3,A})^2\epsilon\cdot\psi=4\epsilon\cdot\psi$ then by Proposition \ref{complexification} we may write $\epsilon=\epsilon_+ +\epsilon_-$ with $D^{1/3,A}\epsilon_{\pm}\cdot\psi = \pm2\epsilon\cdot\psi$.  One can easily check that both $\epsilon_+$ and $\epsilon_-$ satisfy $(\dd^A)^\ast\epsilon=0$, and hence by the previous proposition both solve Equation~\eqref{TYM}.  Thus we have proved:
\begin{proposition}
 Let $A$ be an instanton on a principal bundle $\PPP$ over a nearly K\"ahler six-manifold and let $\epsilon\in\Gamma(\Ad_\PPP\otimes T^\ast M)$.  Then $\epsilon$ satisfies the linearised torsionful Yang--Mills Equation~\eqref{TYM} and the gauge-fixing condition $(\dd^A)^\ast\epsilon=0$ if and only if $(D^{1/3,A})^2\epsilon\cdot\psi=4\epsilon\cdot\psi$.
\end{proposition}

\section{Instantons on homogeneous nearly K\"ahler manifolds}
\label{sec:insthomNK}

There are precisely four homogeneous nearly K\"ahler six-manifolds:
\begin{align*}
	S^6 &= G_2/\SU(3),& S^3\times S^3 &= \SU(2)^3/\SU(2),\\
	\CC P^3&=\Sp(2)/\Sp(1)\times\U(1), & F_{1,2,3}&=\SU(3)/\U(1)^2.
\end{align*}
In all four cases, the nearly K\"ahler metric on $G/H$ is induced from a multiple of the Cartan--Killing form on the Lie algebra $\gg$ of $G$.  In particular, the metric normalised as in Section \ref{sec:geomNK} is induced from the positive symmetric bilinear form \cite{MS10},
\[ B(X,Y) = -\frac{1}{12}\Tr_\gg(\ad(X)\ad(Y))\quad\forall X,Y \in \gg. \]

Let $\mm$ denote the orthogonal complement with respect to $B$ of the Lie algebra $\hh$ of $H$.  The subspace $\mm$ is invariant under the adjoint action of $H$, making the homogeneous space reductive.  The tangent and cotangent bundles of $G/H$ may be identified with the bundles associated to the $H$-principal bundle $G\to G/H$ via the natural representations
\[ \rho_\mm\colon H\to GL(\mm),\quad \rho_{\mm*}\colon H\to GL(\mm^*). \]

The canonical connection on the $H$-principal bundle $G\to G/H$ is by definition the $\hh$-valued part of the left-invariant Maurer--Cartan form on $G$.  The curvature of this connection is $G$-invariant, and may be identified with the $H$-invariant element of $\Lambda^2\mm^\ast \otimes \hh$ given by
\begin{equation}
\label{CC curvature}
 F(X,Y) = - \pi_{\hh}([X,Y]) \quad \forall X,Y\in\mm
\end{equation}
(here $\pi_\hh$ denotes projection onto $\hh$).  The canonical connection on the principal bundle induces a connection on the tangent bundle, and it is well-known
 that this connection coincides with the canonical connection associated with the nearly K\"ahler structure. 
 To verify this fact, it suffices to verify that the holonomy is contained in $\SU(3)$ and that the torsion is skew-symmetric, and then appeal to \cite[Thm 10.1]{Friedrich-Ivanov-skewtorsion}.  Alternatively, one could further verify that the torsion is parallel and appeal to \cite[Thm 4.2]{Agricola:SrniLectures}. 
  In particular, this connection is an instanton.

The four nearly K\"ahler coset spaces therefore provide an ideal testing ground for the study of nearly K\"ahler instanton deformations.  There are in fact two natural deformation problems to consider, as the canonical connection provides a connection on both the $H$-principal bundle $G \to G/H$ and the $\SU(3)$-principal bundle associated with the $\SU(3)$-structure.  The remainder of this article is devoted to answering the following two questions on the four nearly K\"ahler coset spaces $G/H$:
\begin{enumerate}
 \item Does the canonical connection admit any deformations as an instanton with gauge group contained in $H$?
 \item Does the canonical connection admit any deformations as an instanton with gauge group contained in $\SU(3)$?
\end{enumerate}
To answer these questions, we solve the equation $D^{1/3}(\epsilon\cdot\psi)=2\epsilon\cdot\psi$ for  a section  $\epsilon$ of the bundle associated to the $H$-representation $\hh\otimes\mm^\ast$ in the first case and $\su(3)\otimes\mm^\ast$ in the second case.  Note that a positive answer to the first question implies a positive answer to the second.  For the case of $\SU(3)/\U(1)^2$, the first question has already been answered in the negative in Theorem \ref{abelian rigid} using a simple positivity argument.  In the remainder of this section we show that the same argument is inapplicable for the remaining three coset spaces, and in the following section complete answers are derived using group-theoretical analysis.

To this end, we first present a formula for the $F$-dependent term in Equation~\eqref{SLNK} in terms of a Casimir operator for $\hh$.  We define $\Cas_\hh\in\Sym^2(\hh)$ to be the inverse of the metric on $\hh$ obtained by restriction of $B$.  If $I_1,\ldots,I_{\dim(H)}$ is an orthonormal basis for $\hh$ then
\[ \Cas_\hh = \sum_{i=1}^{\dim(H)} I_i\otimes I_i. \]
If $\rho$ is any representation of $H$ we write $\rho(\Cas_{\hh}) = \sum_{i=1}^{\dim(H)}\rho(I_i)\rho(I_i)$.  (Here and throughout we denote by the same symbol representations of a Lie group and its Lie algebra).
\begin{lemma}\label{casimir curvature}
 Let $(E,\rho_E)$ be any representation of $H$.  Let $F\in \Lambda^2\mm^\ast \otimes \hh$ be as in Equation~\eqref{CC curvature} and let $\epsilon\in \mm^\ast\otimes E$.  Then
 \begin{equation}
 	 -2 F \llcorner \epsilon = (\rho_{\mm^\ast}(\Cas_\hh)\otimes 1_E + 1_{\mm^\ast}\otimes \rho_E(\Cas_\hh) - \rho_{\mm^\ast\otimes E} (\Cas_\hh) ) \epsilon.\label{eqn:casimir curvature}
 \end{equation}
\end{lemma}
\begin{proof}
 Pick orthonormal bases $I_i$ with $i\in\{1,\ldots,\dim(H)\}$ for $\hh$ and $I_a$ with $a\in\{\dim(H)+1,\ldots,\dim(H)+6\}$ for $\mm$. To simplify notation, we use indices $i,j,k\in\{1,\ldots,\dim(H)\}$ and $a,b,c\in\{\dim(H)+1,\ldots,\dim(H)+6\}$.  Let $e^a$ be the basis for $\mm^\ast$ dual to $I_a$.  The structure constants $f_{ij}^k,f_{ia}^b,f_{ab}^i,f_{ab}^c$ are defined by the formulae
 \[ [I_i,I_j]=f_{ij}^kI_k,\quad [I_i,I_a] = f_{ia}^b I_b,\quad [I_a,I_b] = f_{ab}^i I_i+ f_{ab}^c I_c. \]
 Then
 \[ \rho_{\mm}(I_i)I_a = f_{ia}^b I_b ,\]
 and
 \[ \rho_{\mm^\ast}(I_i)e^a = -f_{ib}^a e^b .\]
 The expression for $F$ in components is $F=-\frac12 f^i_{ab} e^a\wedge e^b I_i$.  Let us write $\epsilon=e^a\otimes\epsilon_a$, with $\epsilon_a\in E$.  Then
 \begin{align*}
  -2F\llcorner \epsilon &= f^i_{ab}(e^a\wedge e^b)\llcorner e^c\otimes\rho_E(I_i) \epsilon_c \\
                        &= 2f^i_{ab}e^a\otimes\rho_E(I_i) \epsilon_b \\
                        &= 2f^b_{ia}e^a\otimes\rho_E(I_i) \epsilon_b \\
                        &= -2\rho_{\mm^\ast} (I_i) \otimes \rho_E(I_i) \epsilon.
 \end{align*}
 Now $\rho_{\mm^\ast\otimes E}(I_i) = \rho_{\mm^\ast}(I_i)\otimes 1_E + 1_{\mm^\ast}\otimes \rho_E(I_i)$, so
 \[ \rho_{\mm^\ast\otimes E}(\Cas_\hh) = \rho_{\mm^\ast}(\Cas_\hh)\otimes 1_E + 1_{\mm^\ast}\otimes \rho_E(\Cas_\hh) + 2\rho_{\mm^\ast} (I_i) \otimes \rho_E(I_i). \]
 The result follows.
\end{proof}

For the positivity argument used in Theorem \ref{abelian rigid} to be applied to any other instanton, it is necessary that the curvature term in Equation \eqref{SLNK2} is greater than $-4$.  The following proposition shows that this condition does not hold for any of the homogeneous nearly K\"ahler manifolds other than $\SU(3)/\U(1)^2$.

\begin{proposition}
 \label{curvature eigenvalues}
 Let $M=G/H$ be a homogeneous nearly K\"ahler manifold and let $A$ be the canonical connection on $TM$.  Then the operator $\epsilon\mapsto -2F\llcorner\epsilon$ on $\mm^\ast\otimes \hh$ has the following eigenvalues and eigenspace dimensions:
 \begin{itemize}

  \item case $G_2/\SU(3)$:
  \begin{center}\begin{tabular}{c|ccc}
    eigenvalue & $-9$ & $-3$ & $3$ \\ \hline
    dimension & $6$ & $12$ & $30$
  \end{tabular}\end{center}

  \item case $\SU(2)^3/\SU(2)$:
  \begin{center}\begin{tabular}{c|ccc}
    eigenvalue & $-8$ & $-4$ & $4$ \\ \hline
    dimension & $2$ & $6 $& $10$
  \end{tabular}\end{center}

  \item case $\Sp(2)/\Sp(1)\times\U(1)$:
  \begin{center}\begin{tabular}{c|ccc}
    eigenvalue & $-8$ & $0$ & $4$ \\ \hline
    dimension & $4$ & $12$ & $8$
  \end{tabular}\end{center}
 \end{itemize}
\end{proposition}

\begin{proof}
 In all cases the operator is evaluated using the Casimir expression given in Lemma \ref{casimir curvature}, with $E=\hh$ or an irreducible subspace thereof.
 We use the Freudenthal formula for $\Cas_\hh$, which states that (for any Lie algebra $\hh$)
 \begin{equation}
 \label{freudenthal}
 \rho_\lambda (\Cas_{\hh}) = B(\lambda,\lambda) + 2 B(\lambda,\delta)
 \end{equation}
 in the irreducible representation with highest weight $\lambda$, with $\delta$ equal to half of the sum of the positive roots of $\hh$ (see for instance \cite[p.~122]{Humphreys}).

\medskip
\noindent\textbf{Case: $G_2/\SU(3)$}
\medskip\newline
We first choose the Cartan subalgebra of $\su(3)$ consisting of diagonal matrices and call a weight positive if it evaluates to a positive quantity on $\diag(1,0,-1)$.  The matrices $H_1=\diag(1,-1,0)$ and $H_2=\diag(0,1,-1)$ are dual to the fundamental weights $\lambda_1$ and $\lambda_2$.  A word of caution is advisable here and as one computes the similar quantity in the other cases: $\hh$ is \emph{not} a Cartan subalgebra of $\gg$, it is the Lie algebra of $H$ and $B$ is \emph{not} $(-\sfrac1{12})^{\mathrm{th}}$ of the Cartan--Killing form of $\hh$: the trace is taken over all of $\gg$. One may verify by direct calculation that
\begin{equation}\label{eqn:Bsu3}
	\vect{B(H_1,H_1) & B(H_1,H_2)\\ B(H_2,H_1)& B(H_2,H_2)}=\vect{ -4/3 &  \phantom{-}2/3 \\  \phantom{-}2/3& -4/3}.
\end{equation}
Hints are provided in Appendix \ref{sec:_g_2}.    It follows that 
 \[\vect{B(\lambda_1,\lambda_1) & B(\lambda_1,\lambda_2)\\ B(\lambda_2,\lambda_1)& B(\lambda_2,\lambda_2)}=\vect{ -1 & -1/2 \\  -1/2& -1}.\]
  One finds that $\delta=\lambda_1+\lambda_2$.  Therefore, in the complex representation $(V_{(m_1,m_2)},\rho_{(m_1,m_2)})$ with highest weight $m_1\lambda_1+m_2\lambda_2$ one finds that
 \begin{equation}
 	\label{eqn:CasimirSU3}  \rho_{(m_1,m_2)}(\Cas_\hh) = -(m_1^2+m_2^2+m_1m_2+3m_1+3m_2).
 \end{equation}

The representations that appear in Equation \eqref{eqn:casimir curvature} break up into irreducibles as follows:
\begin{align*}
\mm_\CC^{\ast} &\cong V_{(1,0)}\oplus V_{(0,1)}, \\
\hh_\CC &\cong V_{(1,1)}, \\
\hh_\CC\otimes \mm_\CC^{\ast} &\cong (V_{(1,0)}\otimes V_{(1,1)})\oplus (V_{(0,1)}\otimes V_{(1,1)}) \\
&\cong (V_{(1,0)} \oplus V_{(0,2)} \oplus V_{(2,1)})\oplus( V_{(0,1)} \oplus V_{(2,0)} \oplus V_{(1,2)}).
\end{align*} 
As vector spaces, this decomposition reads
\[\hh_\CC\otimes \mm_\CC^{\ast}\cong \CC^3\oplus\CC^6\oplus\CC^{15}\oplus \CC^3\oplus\CC^6\oplus\CC^{15}\]
and with respect to this decomposition
\[\rho_{\hh_\CC\otimes \mm_\CC^{\ast}}(\Cas_\hh) = \diag(-4, -10, -16, -4, -10,-16),\]
while
\[\rho_\hh(\Cas_\hh)= -9\quad \text{ and }\quad\rho_{\mm_\CC^{\ast}}(\Cas_\hh)=-4.\]
The result follows by adding these numbers as dictated by Lemma \ref{casimir curvature}.

\medskip
\noindent\textbf{Case: $\SU(2)^3/\SU(2)$}
\medskip\newline
\label{case SU2 3  mod SU2}%
Let $J_1, J_2, J_3$ be a basis for $\su(2)$ satisfying $[J_i,J_j]=\epsilon_{ijk} J_k$.  A basis for the diagonal subalgebra $\hh$ of $\su(2)\oplus\su(2)\oplus\su(2)$ is given by $I_i=(J_i,J_i,J_i)$ for $i=1,2,3$.  By direct calculation (see for instance Appendix \ref{sec:_su_2}), one finds that $B(I_i, I_j)=\sfrac12\delta_{ij}$, so $\Cas_\hh=2\sum_{i=1}^3I_i\otimes I_i$.  Denote by $(V_m,\rho_m)$ the $(m+1)$-dimensional complex irreducible representation of $\hh\cong\su(2)$; then \begin{equation}\label{eqn:rhomCas su2}
	\rho_m(\Cas_\hh)=-\sfrac12 m(m+2).
\end{equation}

The representations that appear in Equation \eqref{eqn:casimir curvature} break up into irreducibles as follows:
\begin{align*}
\mm_\CC^{\ast} &\cong 2V_2, \\
\hh_\CC &\cong V_2, \\
\hh_\CC\otimes \mm_\CC^{\ast} &\cong 2V_0\oplus 2V_2\oplus 2V_4.
\end{align*}
As vector spaces, this decomposition reads
\[\hh_\CC\otimes \mm_\CC^{\ast}\cong \CC^2\oplus\CC^6\oplus\CC^{10}\]
and with respect to this decomposition
\[\rho_{\hh_\CC\otimes \mm_\CC^{\ast}}(\Cas_\hh) = \diag(0,-4,-12),\]
while
\[\rho_{\hh_\CC}(\Cas_\hh)= -4\quad \text{ and }\quad\rho_{\mm_\CC^{\ast}}(\Cas_\hh)=-4.\]
The result now follows by adding these numbers as dictated by Lemma \ref{casimir curvature}.
 
\medskip
\noindent\textbf{Case: $\Sp(2)/\Sp(1)\times\U(1)$}
\medskip\newline
A basis for $\hh=\sp(1)\oplus\uu(1)$ is given by
 \[ I_1=\begin{pmatrix} \ii & 0 \\ 0 & 0 \end{pmatrix},\quad I_2=\begin{pmatrix} \jj & 0 \\ 0 & 0 \end{pmatrix},\quad I_3=\begin{pmatrix} \kkk & 0 \\ 0 & 0 \end{pmatrix},\quad I_4=\begin{pmatrix} 0 & 0 \\ 0 & \ii \end{pmatrix}. \]
 By direct calculation one finds that $B(I_i, I_j) = \delta_{ij}$.
Since $\su(2)\cong \sp(1)$, we can reuse some of the previous computation.  Let's use notation with prime ($I_1',I_2',I_3',B',\Cas'$) to denote the previous case.  Since $[\ii,\jj]=2\kkk$, we have $[I_i,I_j]=2\epsilon_{i,j,k}I_k$, hence $I_i'=\frac12 I_i$.
Also, we know that $B'(I_i',I_j')=\frac12\delta_{ij}$ while  $B(I_i,I_j)=\delta_{ij}$,
so
\[B(I_i,I_j)=\delta_{ij}=2B'(I_i',I_j')=\frac12B'(I_i,I_j),\]
hence $B=\frac12 B'$, hence $\Cas_{\sp(1)}=2\Cas'=4\sum_{i=1}^3I_i'\otimes I_i'=\sum_{i=1}^3I_i\otimes I_i$.
We therefore have that $\Cas_\hh=\sum_{i=1}^4 I_i\otimes I_i$.  Let $(V_{(m,n)},\rho_{(m,n)})$ denote the unique irreducible $(m+1)$-dimensional complex representation of $\sp(1)\oplus\uu(1)$ in which $\rho_{(m,n)}(I_4)=n\ii$.  Then $\rho_{(m,n)}(\Cas_\hh)$ has eigenvalue $-m(m+2)-n^2$.

 One finds that
 \begin{align}
  \mm^\ast_\CC &\cong V_{(1,-1)} \oplus V_{(1,1)} \oplus V_{(0,-2)} \oplus V_{(0,2)}, \label{eqn:Sp1xU1 m}\\
  \hh_\CC &\cong V_{(2,0)} \oplus V_{(0,0)}. \label{eqn:Sp1xU1 h}
 \end{align}
 The Casimir has a single eigenvalue $-4$ on $\mm^\ast_\CC$.  On the subspace $V_{(2,0)}$ of $\hh_\CC$ it has eigenvalue $-8$, and on the tensor product
 \begin{equation}
  V_{(2,0)} \otimes \mm^\ast_\CC \cong V_{(1,-1)}\oplus V_{(3,-1)}\oplus V_{(1,1)}\oplus V_{(3,1)}\oplus V_{(2,-2)}\oplus V_{(2,2)} \label{eqn:Sp1xU V20tensor m}
 \end{equation}
 it has eigenvalues $-4, -12$ and $-16$ with eigenspaces of dimension $4$, $6$ and $8$ respectively.  On the subspace $V_{(0,0)}$ of $\hh_\CC$ it has eigenvalue zero and on the tensor product $V_{(0,0)}\otimes \mm^\ast_\CC\cong\mm^\ast_\CC$ it has eigenvalue $-4$.  The result then follows from Lemma \ref{casimir curvature}.
\end{proof}

In \cite{Xu09} the eigenvalues of an operator proportional to $\epsilon\mapsto -2F\llcorner \epsilon$ are calculated for the canonical connection on $S^6$.  It was claimed that all eigenvalues were non-negative, leading to the conclusion that this instanton is rigid.  The previous proposition shows that, on the contrary, this operator has both negative and positive eigenvalues.

Support for the accuracy of our calculation is provided by the following observation.  Since both the action of two-forms by contraction on 1-forms and the adjoint action of $\su(3)$ are traceless, the operator $\epsilon\mapsto -2F\llcorner \epsilon$ must be traceless.  The result presented in Proposition \ref{curvature eigenvalues} is consistent with the tracelessness of this operator.  In contrast, the calculation leading to \cite[Lemma 3.4]{Xu09} implies that this operator is not traceless, so cannot be correct.  In fact, any traceless non-negative operator is necessarily zero, so this curvature operator is non-negative only in trivial cases.

Thus the proof of rigidity of the instanton on $S^6$ given by \cite[Thm 3.3]{Xu09} is invalid.  For reasons explained above we have not been able to prove the rigidity of the instanton on $S^6$ using a positivity argument along the lines of \cite{Xu09}.  In the next section we prove the rigidity of this instanton using more powerful methods.

\section{The spectrum of the Laplacian}
\label{sec:frobenius}

In this section the space of deformations of the canonical connection on each of the homogeneous nearly K\"ahler manifolds is determined.  To compute this space, we first derive a representation-theoretic expression for the whole of the operator on the right hand side of the Schr\"odinger--Lichnerowicz formula given by Equation \eqref{SLNK2}, and then determine its spectrum using Frobenius reciprocity and standard formulae for the eigenvalues of the quadratic Casimir in particular representations.

Let $E$ be any representation of $H\subset G$, and let $L^2(G;\mm^\ast\otimes E)$ denote the $L^2$ completion of the space of $\mm^\ast\otimes E$-valued functions on $G$.  This linear space carries the following left representations of $G$:
\begin{itemize}
 \item The left regular representation $f\mapsto \rho_L(g)f$, where
 \[ \rho_L(g)f(g'):= f(g^{-1}g')\quad \forall g,g'\in G, f\in L^2(G;\mm^\ast\otimes E); \]
 \item The right regular representation $f\mapsto \rho_R(g)f$, where
 \[ \rho_R(g)f(g'):= f(g'g)\quad \forall g,g'\in G, f\in L^2(G;\mm^\ast\otimes E); \]
 \item The representation induced by the representation $\rho_{\mm^\ast}$ on $\mm^\ast$ (denoted by the same symbol $\rho_{\mm^\ast}$);
 \item The representation induced by the representation $\rho_{E}$ on $E$ (denoted by the same symbol $\rho_E$).
\end{itemize}
The $L^2$ completion of the space of sections of the vector bundle $T^\ast (G/H)\otimes P\times_H E$ can be identified with the fixed set $L^2(G;\mm^\ast\otimes E)_H\subset L^2(G;\mm^\ast\otimes E)$ of the combined actions of $\rho_R$, $\rho_E$ and $\rho_{\mm^\ast}$.  On the level of Lie algebras, elements of $L^2(G;\mm^\ast\otimes E)_H$ are functions $f$ satisfying
\begin{equation}\label{eq:4.1} \rho_R(X)f + \rho_{\mm^\ast\otimes E} (X) f = 0 \qquad\forall X\in \hh. \end{equation}

\begin{proposition}
 Let $A$ be the canonical connection on a nearly K\"ahler coset space $G/H$, let $F$ be its curvature, let $\psi$ be the Killing spinor, let $(\rho_E,E)$ be a representation of $H$, and let $\epsilon$ be a smooth section of $T^\ast (G/H)\otimes E(G/H)$.  Then
 \[ (D^{1/3,A})^2 \epsilon\cdot\psi = (-\rho_L(\Cas_\gg)\epsilon + \rho_E(\Cas_\hh)\epsilon + 4\epsilon)\cdot\psi. \]
\end{proposition}
\begin{proof}
 By the Schr\"odinger--Lichnerowicz formula \eqref{SLNK2} the square of the Dirac operator can be expressed as a sum of three terms involving a Laplacian, a curvature operator, and scalar multiplication.
 
 The covariant derivative $\nabla^{1,A}$ from which the Laplacian is built is equal to the covariant derivative on $\mm^\ast\otimes E$ induced by the canonical connection of the homogeneous space.  It is a standard result (see \cite{MS10}) that the Laplacian can be identified with the action of a Casimir on $C^\infty(G;\mm^\ast\otimes E)_H$:
 \[ (\nabla^{1,A})^\ast\nabla^{1,A} \epsilon = - \rho_R(\Cas_\mm)\epsilon. \]
 (Note that the right action of $\Cas_\mm$ on $C^\infty(G;\mm^\ast\otimes E)$ descends to an action on $C^\infty(G;\mm^\ast\otimes E)_H$ because $\Cas_\mm$ is $H$-invariant).

 The curvature term may be expressed as a sum of Casimirs by virtue of Lemma \ref{casimir curvature}.  Inserting these expressions into the Schr\"odinger--Lichnerowicz formula \eqref{SLNK2} yields
 \[ (D^{1/3,A})^2 (\epsilon\cdot\psi) = \big( - \rho_R(\Cas_\mm) -\rho_{\mm^\ast\otimes E}(\Cas_\hh) + \rho_{\mm^\ast}(\Cas_\hh)+\rho_E(\Cas_\hh)+8\big) \epsilon\cdot\psi. \]

 Equation \eqref{eq:4.1} implies that $\rho_{\mm^\ast\otimes E} (\Cas_\hh)\epsilon = \rho_R(\Cas_\hh)\epsilon$ for every $\epsilon\in C^\infty(G;\mm^\ast\otimes E)_H$.  Moreover,
 \[ - \rho_R(\Cas_\mm) - \rho_R(\Cas_\hh) = -\rho_L(\Cas_\gg), \]
 because the left and right actions of $\Cas_\gg=\Cas_\hh+\Cas_\mm$ on $C^\infty(G;\mm^\ast\otimes E)$ agree.  Finally, the operator $\rho_{\mm^\ast}(\Cas_\hh)$ is equal to minus the action of the Ricci curvature of the canonical connection on the cotangent bundle. 
  Proposition \ref{prop:Ricci} tells us the Ricci curvature is equal to 4 times the identity.  Combining the above observations yields the advertised result.
\end{proof}

The previous proposition allows direct verification of the rigidity of the canonical connections on the nearly K\"ahler homogeneous spaces.  Deformations correspond to eigenfunctions $\epsilon\cdot\psi$ of the Dirac operator with eigenvalue 2.  Any such section satisfies $(D^{1/3,A})^2(\epsilon\cdot\psi)=4\epsilon\cdot\psi$, or equivalently,
\begin{equation}\label{casimir equality} \rho_L(\Cas_\gg)\epsilon = \rho_E(\Cas_\hh)\epsilon. \end{equation}
This equation may be solved using the Frobenius reciprocity theorem (see for instance \cite[Thm 1.14]{Knapp-reptheory}) by following this algorithm:
\begin{enumerate}
\item Identify the complex representation $E$ of $H$ associated with the principal bundle under question.  In the case of the $H$-bundle $G\to G/H$, $E=\hh_\CC:=\hh\otimes\CC$ with its adjoint action.  In the case of the SU(3)-bundle associated with the tangent bundle, note that the action of $H$ on $\mm^\ast$ defines a homomorphism $H\to \SU(3)\subset\End(\mm^\ast)$.  Then $E=\su(3)_\CC$ and $H$ acts on this space adjointly.
\item Identify the irreducible components $E_\alpha$ of the representation $E$ of $H$, so that
\[(E,\rho_E)=\bigoplus_\alpha n_\alpha(E_\alpha,\rho_\alpha).\]
\item For each $\alpha$, calculate the eigenvalue \[C_\alpha=\rho_\alpha(\Cas_\hh)\] of $\Cas_\hh$ acting on $E_\alpha$.
\item  Identify all irreducible representations $(W_{\alpha\gamma},\sigma_{\alpha\gamma})$ of $G$ for which $\sigma_{\alpha\gamma}(\Cas_\gg)=C_\alpha$.
\item For each $\alpha$ identify the decomposition $E_\alpha\otimes \mm_\CC^\ast=\bigoplus_{\beta} U_{\alpha\beta}$ into irreducible components as representations of $H$.  For each $(\alpha,\beta,\gamma)$, determine the multiplicity $n(\alpha,\beta,\gamma)$ of each such $U_{\alpha\beta}$ in the representation $(W_{\alpha\gamma},\sigma_{\alpha\gamma}|_H)$ of $H$.
\end{enumerate}
We then have that
\begin{equation*}
	L^2(G;E\otimes \mm_\CC^\ast)_H=\bigoplus_{\alpha,\beta}n_\alpha L^2(G;U_{\alpha\beta})_H
\end{equation*}
and Frobenius reciprocity tells us that the multiplicity of $(W_{\alpha\gamma},\sigma_{\alpha\gamma})$ in $L^2(G;U_{\alpha\beta})_H$ as representations of $G$ is equal to the multiplicity $n(\alpha,\beta,\gamma)$ of $U_{\alpha\beta}$ in $W_{\alpha\gamma}$ as representations of $H$.  
Since $L^2(G;E\otimes\mm^\ast_\CC)_H=\bigoplus_{\alpha,\beta}n_\alpha L^2(G;U_{\alpha\beta})_H$
and since by construction, the space of solutions to Equation \eqref{casimir equality} in $L^2(G;U_{\alpha\beta})_H$ isomorphic to $\bigoplus_{\gamma}n(\alpha,\beta,\gamma)W_{\alpha\gamma}$, 
 the space of solutions to Equation \eqref{casimir equality} in $L^2(G;E\otimes\mm^\ast_\CC)_H$ is isomorphic to
\[ \bigoplus_{\alpha,\beta,\gamma} n_\alpha n(\alpha,\beta,\gamma) W_{\alpha\gamma}. \]

This vector space is the complexification of $\ker\bigl((D^{1/3,A})^2-4\bigr)\cap\Omega^1L_0$, so $\ker\bigl((D^{1/3,A})^2-4\bigr)\cap\Omega^1L_0$ is a real representation of $G$ whose complexification is isomorphic to the above sum of irreducibles.  By Proposition \ref{complexification}, $\ker\bigl((D^{1/3,A})^2-4\bigr)\cap\Omega^1L_0$ is isomorphic to two copies of $\ker\bigl(D^{1/3,A}-2\bigr)\cap\Omega^1L_0$.  Therefore we have proved the following result.

\begin{lemma}\label{lemma:deformations}
	The space of instanton perturbations is a real representation of $G$ whose complexification is isomorphic to
\[ \bigoplus_{\alpha,\beta,\gamma} \frac{n_\alpha n(\alpha,\beta,\gamma)}{2} W_{\alpha\gamma}. \]
\end{lemma}

\begin{theorem}
\label{thm:explicit deformations}
Let $G/H$ be one of the four homogeneous six-dimensional nearly K\"ahler manifolds and let $A$ be its canonical connection.  The spaces of deformations of $A$ within the $H$-principal bundle $G\to G/H$ are isomorphic to the following representations of $G$:
\begin{center}\begin{tabular}{c|c|c|c}
$G_2/\SU(3)$ & $\SU(2)^3/\SU(2)$ & $\Sp(2)/\Sp(1)\times\U(1)$ & $\SU(3)/\U(1)^2$ \\
\hline
$0$ & $0$ & $W_{(1,0)}^\RR$ & $0$
\end{tabular}\end{center}
In the notation to be introduced below, $W_{(1,0)}^\RR$ is the real representation of $\Sp(2)$ whose complexification is irreducible with highest weight $(1,0)$; it is the unique five-dimensional real irreducible representation.

The spaces of deformations of $A$ within the $\SU(3)$-principal bundle associated with the $\SU(3)$-structure on $G/H$ are isomorphic to the following representations of $G$:
\begin{center}\begin{tabular}{c|c|c|c}
$G_2/\SU(3)$ & $\SU(2)^3/\SU(2)$ & $\Sp(2)/\Sp(1)\times\U(1)$ & $\SU(3)/\U(1)^2$ \\
\hline
$0$ & $\gg$ & $W_{(1,0)}^\RR\oplus2\gg$ & $6\gg$
\end{tabular}\end{center}
In this table $\gg$ denotes the appropriate adjoint representation in each column.
\end{theorem}

\begin{proof}
We proceed case by case to follow the algorithm explained above.
\medskip\newline
\textbf{Case: }$G_2/\SU(3)$ with structure group $H$
\medskip\newline
The adjoint representation $E= \su(3)$ of the gauge group $\SU(3)$ is irreducible.  As noted in the proof of Proposition \ref{curvature eigenvalues}, the unique eigenvalue of $\rho_E(\Cas_\hh)$ on this space is $-9$.

In order to calculate eigenvalues of $\Cas_{\gg_2}$ we appeal once again to the Freudenthal formula \eqref{freudenthal}.  The Cartan subalgebra for the complexified $(\gg_2)_\CC$ is in fact the Cartan subalgebra of $\sll_3\CC\subset (\gg_2)_\CC$.  It is spanned by the matrices $H_1=\diag(1,-2,1)$ and $H_2=\diag(0,1,-1,)$ \label{H1H2G2} dual to the fundamental weights $\lambda_1$ and $\lambda_2$ of $\gg_2$.  By direct calculation one finds that
\begin{equation}
\vect{B(H_1,H_1) & B(H_1,H_2)\\ B(H_2,H_1)& B(H_2,H_2)}=\vect{ -4 &  \phantom{-}2 \\  \phantom{-}2& -\frac43}.	\label{eqn:BG2}
\end{equation}
 Therefore 
 \[\vect{B(\lambda_1,\lambda_1) & B(\lambda_1,\lambda_2)\\ B(\lambda_2,\lambda_1)& B(\lambda_2,\lambda_2)}=\vect{ -1 & -\frac32 \\  -\frac32& -3}.\]
 One finds that $\delta$, defined to be half of the sum of the positive roots, is equal to $\lambda_1+\lambda_2$\label{deltaG2}; see indeed Equation \eqref{eqn:deltaG2}.  Therefore the Freudenthal formula of Equation \eqref{freudenthal} yields
\begin{align*} \nu_{(m_1,m_2)}(\Cas_{\gg_2})&=B(m_1\lambda_1+m_2\lambda_2,m_1\lambda_1+m_2\lambda_2)+2B(m_1\lambda_1+m_2\lambda_2,\lambda_1+\lambda_2)\\
	&=-(m_1^2+3m_2^2+3m_1m_2+5m_1+9m_2)
	\end{align*}
for the Casimir in the representation $\nu_{(m_1,m_2)}$ of $(\gg_2)_\C$ with highest weight $m_1\lambda_1+m_2\lambda_2$.
 The smallest eigenvalues for the Casimir operator are 0, $-6$ and $-12$.  In particular $-9$ is not an eigenvalue of $\Cas_{\gg_2}$ in any representation.  Therefore Equation~\eqref{casimir equality} admits no solutions $\epsilon\in C^\infty(G_2;\mm^\ast\otimes\su(3))_{\SU(3)}$ and the canonical connection admits no perturbations.
 
\medskip\noindent
\textbf{Case: }$G_2/\SU(3)$ with structure group $\SU(3)$
\medskip\newline
This case is identical to the case of structure group $H$, since $H=\SU(3)$.

\medskip\noindent
\textbf{Case: }$\SU(2)^3/\SU(2)$ with structure group $H$
\medskip\newline
The adjoint representation $E=\su(2)$ of $H=\SU(2)$ is irreducible, and, as has already been noted in the proof of Proposition \ref{curvature eigenvalues}, the unique eigenvalue of $\rho_E(\Cas_\hh)$ on this space is $-4$.

 A basis for $\gg=\su(2)\oplus\su(2)\oplus\su(2)$ is given by $I_i^{(1)}=(J_i,0,0)$, $I_i^{(2)}=(0,J_i,0)$ and $I_i^{(3)}=(0,0,J_i)$, where $i=1,2,3$ and $J_i$ are a basis for $\su(2)$ satisfying $[J_i,J_j]=\epsilon_{ijk}J_k$.  One finds that $B(I_i^{(a)},I_j^{(b)})=\frac16\delta_{ij}\delta_{ab}$, and hence that $\Cas_\gg=6\sum_{i,a} I_i^{(a)}\otimes I_i^{(a)}$.  Irreducible representations of $\gg$ take the form $V^{(1)}_{m_1}\otimes V^{(2)}_{m_2}\otimes V^{(3)}_{m_3}$, where $V^{(a)}_{m}$ denotes the irreducible representation of the $a$-th copy of $\su(2)$ with highest weight $m$.  The Casimir is equal to $-\frac32\sum_a m_a(m_a+2)$ on such a representation.  The following table gives the smallest eigenvalues of $\Cas_\gg$ amongst all representations:
 \begin{center}
 \begin{tabular}{c|cccc}
  $(m_1,m_2,m_3)$ & (1,0,0) & (1,1,0) & (2,0,0) & (1,1,1) \\
  \hline
  $\rho_{(m_1,m_2,m_3)}(\Cas_\gg)$ & \rule{0pt}{19pt}$\displaystyle-\frac92$ & $-9$ & $-12$ & $\displaystyle-\frac{27}{2}$
 \end{tabular}
 \end{center}
 Since $-4$ does not occur as an eigenvalue of the $\Cas_\gg$, Equation~\eqref{casimir equality} admits no solutions on $L^2(G;\su(2)\otimes \mm^\ast)_H$ and the instanton is stable.

\medskip\noindent
\textbf{Case: }$\SU(2)^3/\SU(2)$ with structure group $\SU(3)$
\medskip\newline
Next we consider the case of gauge group $\SU(3)$.  We now give a different basis of $\gg=\su(2)\oplus\su(2)\oplus\su(2)$.  As on page \pageref{case SU2 3  mod SU2}, we let $I_i=I_i^{(1)}+I_i^{(2)}+I_i^{(3)}=(J_i,J_i,J_i)$ be the basis of $\hh=\su(2)$.  The orthogonal complement $\mm$ can either be seen as a  six-dimensional real vector space with basis $X_i=((1+\sqrt2)J_i,(1-\sqrt2)J_i,-2J_i)$ and $Y_i=\sqrt{6}(J_i,-J_i,0)$ or a three-dimensional complex vector space with basis $X_1,X_2,X_3$.   The almost complex structure sends $X_i$ to $Y_i$ and $Y_i$ to $-X_i$.

The action of $\su(2)$ on $\mm$ defines a homomorphism $\su(2)\to\su(3)$ where $\ad(I_i)(X_j)=\epsilon_{ijk}X_k$.  Under the adjoint action of $\su(2)$, $E=\su(3)$ breaks up into two irreducible pieces:
\begin{equation}
	\label{eqn:su3V2V4} \su(3)_\CC \cong V_2\oplus V_4.
\end{equation}
 The component $V_2$ is just the embedded $\su(2)\subset\su(3)$.  We have already shown that Equation~\eqref{casimir equality} admits no solutions in $L^2(G;V_2\otimes \mm^\ast)_H$.  It remains to investigate the same equation on $L^2(G;V_4\otimes \mm^\ast)_H$.  On the subspace $V_4\subset \su(3)$ one has
 \[ \rho_E(\Cas_\hh) = -12. \]
 The Casimir $\Cas_\gg$ attains eigenvalue $-12$ precisely in the irreducible representations $V_{(2,0,0)}$, $V_{(0,2,0)}$ and $V_{(0,0,2)}$.  It remains to determine whether these occur as subrepresentations of $L^2(G;V_4\otimes \mm^\ast)_H$.

 As representation of $H$, we decompose $V_4\otimes \mm^\ast$  into irreducible pieces as
 \[ V_4\otimes \mm^\ast_\CC \cong 2V_2 \oplus 2V_4 \oplus 2V_6. \]
 The restriction of the representation $V_{(2,0,0)}$ of $G$ to $H$ is isomorphic to $V_2$.  Therefore by Frobenius reciprocity $V_{(2,0,0)}$ occurs in $L^2(G;V_2)_H$ with multiplicity $1$ and does not occur in $L^2(G;V_4)_H$ or $L^2(G;V_6)_H$.  Similarly, the representations $V_{(0,2,0)}$ and $V_{(0,0,2)}$ of $G$ each occur in $L^2(G;V_2)_H$ with multiplicity 1 and do not occur in $L^2(G;V_4)_H$ or $L^2(G;V_6)_H$.  Therefore the set of solutions to Equation~\eqref{casimir equality} in $L^2(G;V_4\otimes\mm^\ast_\CC)_H$ is an 18-dimensional vector space isomorphic to
 \[ 2V_{(2,0,0)}\oplus2V_{(0,2,0)}\oplus2V_{(0,0,2)}. \]
By Lemma \ref{lemma:deformations}, the space of instanton perturbations is isomorphic to a real subspace of
 \[ V_{(2,0,0)}\oplus V_{(0,2,0)}\oplus V_{(0,0,2)}. \]
This representation is isomorphic to the adjoint representation of $G=\SU(2)^3$.

\medskip\noindent
\textbf{Case: }$\Sp(2)/\Sp(1)\times\U(1)$ with structure group $H$
\medskip\newline
As in the proof of Proposition \ref{curvature eigenvalues}, $(V_{(m,n)},\rho_{(m,n)})$  denotes the unique $(m+1)$-dimensional irreducible representation of $\Sp(1)\times\U(1)$ in which $\U(1)$ acts with weight $n$.  Again, $\rho_{(m,n)}(\Cas_\hh)=-m(m+2)-n^2$. 

The adjoint representation $E=\hh_\CC$ splits into irreducible pieces as
 \[ E \cong V_{(2,0)} \oplus V_{(0,0)}. \]
The second component $V_{(0,0)}$ does not give rise to any instanton perturbations since instanton perturbations from this component correspond to perturbations of the part of the instanton with gauge group $\U(1)$, and we have already shown that abelian instantons admit no perturbations.  Therefore we only consider the first component $V_{(2,0)}$, for which
\[  \rho_{(2,0)}(\Cas_\hh)=-8.\]

Now we calculate the eigenvalues of $\Cas_{\gg}$ in irreducible representations of $G=\Sp(2)$.  We choose the Cartan subalgebra of $\sp(2)_\CC$ to be the space of diagonal $2\times2$ skew-adjoint quaternionic matrices of the form $\diag(a\mathbf{i},b\mathbf{i})$, with $a,b\in\CC$ and $\mathbf{i},\mathbf{j},\mathbf{k}$ denoting the imaginary quaternions.  A weight is called positive if it evaluates to a positive real number on the matrix $\ii\,\diag(2\mathbf{i},\mathbf{i})$.  The matrices $H_1=\ii\,\diag(0,\mathbf{i})$ and $H_2=\ii\,\diag(\mathbf{i},-\mathbf{i})$ are dual to the fundamental weights $\lambda_1,\lambda_2$.  By direct calculation one finds that
\[\vect{B(H_1,H_1) & B(H_1,H_2)\\ B(H_2,H_1)& B(H_2,H_2)}=\vect{ -1 &  \phantom{-}1 \\  \phantom{-}1& -2}.\]
 Therefore 
 \[\vect{B(\lambda_1,\lambda_1) & B(\lambda_1,\lambda_2)\\ B(\lambda_2,\lambda_1)& B(\lambda_2,\lambda_2)}=\vect{ -2 & -1 \\  -1& -1}.\]
 The weight $\delta$, defined to be half of the sum of the positive roots, is equal to $\lambda_1+\lambda_2$.  Therefore, by the Freudenthal formula of Equation \ref{freudenthal},
 \begin{align*}
  \sigma_{(m_1,m_2)}(\Cas_\gg) &= B(m_1\lambda_1+m_2\lambda_2,m_1\lambda_1+m_2\lambda_2) + 2B(m_1\lambda_1+m_2\lambda_2,\delta) \\
  &= -(2m_1^2+2m_1m_2+m_2^2+6m_1+4m_2)
 \end{align*}
 in the representation $(W_{(m_1,m_2)},\sigma_{(m_1,m_2)})$ with highest weight $m_1\lambda_1+m_2\lambda_2$.  The following table lists the smallest eigenvalues of $\Cas_\gg$:
 \begin{center}
  \begin{tabular}{c|ccccc}
   $(m_1,m_2)$ & (0,0) & (0,1) & (1,0) & (0,2) & (1,1) \\ \hline
   $\sigma_{(m_1,m_2)}(\Cas_\gg)$\rule{0pt}{12pt} & 0 & $-5$ & $-8$ & $-12$ &$ -15$
  \end{tabular}.
 \end{center}
The only irreducible representation in which $\Cas_\gg$ has eigenvalue $-8$ is $W_{(1,0)}$. 

Next we identify the irreducible subrepresentations of $E\otimes\mm^\ast_\CC$.  Recall now from Equations \eqref{eqn:Sp1xU1 m} and \eqref{eqn:Sp1xU V20tensor m} that as representations $H$, 
\begin{align*}
\mm^\ast_\CC &\cong V_{(1,1)}\oplus V_{(1,-1)}\oplus V_{(0,2)}\oplus V_{(0,-2)}, \text{ and}\\
V_{(2,0)}\otimes\mm^\ast_\CC &\cong V_{(1,1)}\oplus V_{(1,-1)}\oplus V_{(3,1)}\oplus V_{(3,-1)}\oplus V_{(2,2)}\oplus V_{(2,-2)}.
\end{align*}

 A direct computation shows that 
\[W_{(1,0)}^\RR=\left\{\left.\vect{a&z\\ \bar{z}&-a}\hskip1mm \right|\hskip1mm a\in\RR\text{ and }z\in\HH\right\}\]
with $\sigma_{(1,0)}(v)(A)=[v,A]$.  
As representations of $H$, we therefore have
\[ W_{(1,0)} \cong V_{(1,1)}\oplus V_{(1,-1)}\oplus V_{(0,0)}. \]
We note that $W_{(1,0)}$ has two components in common with $V_{(2,0)}\otimes\mm^\ast_\CC$, namely $V_{(1,1)}$ and $V_{(1,-1)}$.  Therefore the space of solutions to Equation \eqref{casimir equality} in $L^2(G;V_{(2,0)}\otimes\mm_\CC^\ast)$ is isomorphic to two copies of $W_{(1,0)}$.  By Lemma \ref{lemma:deformations}, the space of instanton perturbations is isomorphic to the real subspace of $W_{(1,0)}$, which has real dimension 5.

\medskip\noindent
\textbf{Case:} $\Sp(2)/\Sp(1)\times\U(1)$ with gauge group SU(3)
\medskip\newline
It is important here to keep in mind that $(A,a)\in\su(2)\oplus \uu(1)\cong \sp(1)\oplus \uu(1)$  sits in $\su(3)$ as $\diag(A+a,-2a)$.  Hence it acts on $\su(3)$ as
\[\ad((A,a))\vect{B& v\\ -v^*&b} = \vect{[A,B]&(A+3a)v\\ -((A+3a)v)^*&0}.\]
Hence the representation $E:=\su(3)_\CC$ of $H$ splits into irreducible subrepresentations as 
\[ E \cong V_{(2,0)} \oplus V_{(0,0)} \oplus V_{(1,3)}\oplus V_{(1,-3)}. \]

 Perturbations coming from the first two components have already been analysed, and it remains to consider the final two components $V_{(1,3)}$ and $V_{(1,-3)}$.  The eigenvalues of $\Cas_\hh$ on these spaces are both $-12$.  From the analysis of representations of $\sp(2)$ detailed above, we learn that the unique representation of $\gg=\sp(2)$ in which $\Cas_\gg$ has eigenvalue $-12$ is $W_{(0,2)}$, the adjoint representation.

As representations of $H$, we have
\begin{align*}
  V_{(1,3)}\otimes\mm_\CC^\ast &\cong 
  V_{(0,4)}\oplus V_{(2,4)}\oplus V_{(0,2)}\oplus V_{(2,2)} \oplus V_{(1,5)}\oplus V_{(1,1)}, \\
  V_{(1,-3)}\otimes\mm_\CC^\ast &\cong 
  V_{(0,-2)}\oplus V_{(2,-2)}\oplus V_{(0,-4)}\oplus V_{(2,-4)} \oplus V_{(1,-1)}\oplus V_{(1,-5)} .
\end{align*}

Since the representation $W_{(0,2)}$ of $G$ is the adjoint representation, we have that $W_{(0,2)}\cong \hh_\CC\oplus\mm_\CC$ as representations of $H$.  In view of Equations \eqref{eqn:Sp1xU1 m} and \eqref{eqn:Sp1xU1 h}, we have that
\begin{equation}
	W_{(0,2)} \cong V_{(2,0)}\oplus V_{(0,0)}\oplus V_{(1,1)}\oplus V_{(1,-1)}\oplus V_{(0,-2)}\oplus V_{(0,2)}
\end{equation}
as representations of $H$.

This decomposition has two components in common with $V_{(1,3)}\otimes\mm_\CC^\ast$ (namely $V_{(0,2)}$ and $V_{(1,1)}$) and two components in common with $V_{(1,-3)}\otimes\mm_\CC^\ast$ (namely $V_{(0,-2)}$ and $V_{(1,-1)}$).  Therefore $W_{(0,2)}$ occurs in each of $L^2(G;V_{(1,3)}\otimes\mm^\ast_\CC)$ and $L^2(G;V_{(1,-3)}\otimes\mm^\ast_\CC)$ with multiplicity 2.  Taking account of the previous calculation for gauge group $H$, the space of solutions to \eqref{casimir equality} in $L^2(G;E\otimes\mm^\ast_\CC)$ is isomorphic to
\[ 2W_{(1,0)}\oplus 4W_{(0,2)}. \]

By Proposition \ref{complexification}, the space of instanton perturbations is a real dimensional representation of $G$ whose complexification is isomorphic to $W_{(1,0)}\oplus 2W_{(0,2)}$.

\medskip\noindent
\textbf{Case:} $\SU(3)/\U(1)^2$ with gauge group $H$
\medskip\newline
By Theorem \ref{abelian rigid} there are no instanton perturbations with abelian gauge group $\U(1)\times\U(1)$.

\medskip\noindent
\textbf{Case:} $\SU(3)/\U(1)^2$ with gauge group $\SU(3)$
\medskip\newline
In the calculation that follows there are two distinct injective homomorphisms from $\hh=\uu(1)\oplus\uu(1)$ to $\su(3)$.  The first is induced from the inclusion of $\U(1)\times\U(1)$ in the group $\SU(3)$ of isometries.  The second is induced from the inclusion of the structure group $\U(1)\times\U(1)$ of the principal bundle $SU(3) \to \SU(3)/\U(1)\times\U(1)$ into the structure group SU(3) of the tangent bundle.  To distinguish these two homomorphisms, we  denote the target of the first $\gg$ or $\su(3)$ and the target of the second $\widetilde{\su}(3)$.

 The subalgebra $\hh\subset\gg$ is generated by
 \[ H_1=\ii\,\diag(1,-1,0),\quad H_2=\ii\,\diag(0,1,-1). \]
 The complexified orthogonal complement $\mm_\CC$ of $\hh$ splits into two pieces of types (1,0) and (0,1) with respect to the almost complex structure. This complex structure can be given as a function of the 3-symmetry, as in \cite[Eqn.~(5)]{Butruille}.
Given the third root of unity $\zeta=e^{\frac{2\pi i}3}$, the 3-symmetry on $\SU(3)$ is given by conjugating with the clock matrix $\diag(1,\zeta,\zeta^2)$.  The fixed set is the diagonal subgroup $\U(1)\times\U(1)$.  The resulting 3-symmetry $s$ on $\su(3)$ fixes $\hh=\uu(1)\oplus\uu(1)$, and on the orthogonal complement $\mm$, we have the almost complex structure
 $J=\frac{2s+1}{\sqrt{3}}$.  
  A basis for $\mm_\CC^{1,0}$ is given by
\begin{equation}
	 \label{eqn:C1C2C3}C_1=\begin{pmatrix} 0&1&0\\0&0&0\\0&0&0\end{pmatrix},  C_2=\begin{pmatrix} 0&0&0\\0&0&1\\0&0&0\end{pmatrix},  C_3=\begin{pmatrix} 0&0&0\\0&0&0\\1&0&0\end{pmatrix}. 
\end{equation}
 With respect to this basis, the images $\tilde{H}_1,\tilde{H}_2$ of $H_1,H_2$ in $\widetilde{\su}(3):=\su(\mm_\CC^{1,0})$ are
 \[ \tilde{H}_1=\ii\,\diag(2,-1,-1),\quad \tilde{H}_2=\ii\,\diag(-1,2,-1). \]

We denote by $(V_{(m_1,m_2)},\gamma_{(m_1,m_2)})$ the complex one-dimensional representation of $\U(1)\times\U(1)$ such that $\gamma_{(m_1,m_2)}(H_i)=\ii m_i$.  From the considerations above, the representation $E:=\widetilde{\su}(3)_\CC$ of $\U(1)\times\U(1)$ breaks up into irreducibles as 
\[ E \cong 2V_{(0,0)}\oplus V_{(3,0)}\oplus V_{(-3,0)}\oplus V_{(0,3)}\oplus V_{(0,-3)}\oplus V_{(3,-3)}\oplus V_{(-3,3)}, \]
while $\mm_\CC^\ast$ of $H$ breaks up as
\[
\mm_\CC^\ast \cong V_{(2,-1)}\oplus V_{(-1,2)}\oplus V_{(-1,-1)}\oplus V_{(-2,1)}\oplus V_{(1,-2)}\oplus V_{(1,1)}.
\]

The components $V_{(0,0)}$ correspond to perturbations with gauge group $\U(1)^2$, and these have already been analysed.  Therefore we need only consider the remaining six components.

One finds that $B(H_1,H_1)=B(H_2,H_2)=1$ and $B(H_1,H_2)=-1/2$, so the Casimir for $\hh=\uu(1)\oplus\uu(1)$ is
 \[ \Cas_\hh = \frac43 (H_1\otimes H_1 + H_2\otimes H_2) + \frac23(H_1\otimes H_2 + H_2\otimes H_1). \]
Then $\gamma_{(m_1,m_2)}(\Cas_\hh)=-4(m_1^2+m_1m_2+m_2^2)/3$.  Therefore the eigenvalues of $\Cas_\hh$ on $E$ are $0$ (on the components $2V_{(0,0)}$ that have already been analysed) and $-12$ (on the remaining components).

We have yet to compute $\Cas_\gg$.   We use $H_1=\diag(1,-1,0)$ and $H_2=\diag(0,1,-1)$ as on page \pageref{eqn:Bsu3}, but this time the trace involved in the definition of $B$ is over $\sll_3\C$ only, not over $\gg_2$.   We therefore get
\begin{equation}\label{eqn:Bsu3pure}
	\vect{B(H_1,H_1) & B(H_1,H_2)\\ B(H_2,H_1)& B(H_2,H_2)}=\vect{ -1 &  \phantom{-}1/2 \\  \phantom{-}1/2& -1}.
\end{equation}
This $B$ is $3/4$ times the $B$ obtained in Equation \eqref{eqn:Bsu3}, hence the Casimir is $4/3$ times the Casimir of Equation \eqref{eqn:CasimirSU3}).  Thus
 \begin{equation}
 	\label{eqn:CasimirSU3pure}  \rho_{(m_1,m_2)}(\Cas_\gg) = -\frac43(m_1^2+m_2^2+m_1m_2+3m_1+3m_2).
 \end{equation}
We can see that $-12$ is an eigenvalue of $\Cas_\gg$ only in the adjoint representation of $\gg$.  

Therefore the tensor products of each of the irreducible components of $E$ with $\mm^\ast_\CC$ are
 \begin{align*}
  V_{(3,0)}\otimes \mm_\CC^\ast &\cong V_{(5,-1)}\oplus V_{(2,2)}\oplus V_{(2,-1)}\oplus V_{(1,1)}\oplus V_{(4,-2)}\oplus V_{(4,1)}, \\
  V_{(-3,0)}\otimes \mm_\CC^\ast &\cong V_{(-1,-1)}\oplus V_{(-4,2)}\oplus V_{(-4,-1)}\oplus V_{(-5,1)}\oplus V_{(-2,-2)}\oplus V_{(-2,1)}, \\
  V_{(0,3)}\otimes \mm_\CC^\ast &\cong V_{(2,2)}\oplus V_{(-1,5)}\oplus V_{(-1,2)}\oplus V_{(-2,4)}\oplus V_{(1,1)}\oplus V_{(1,4)}, \\
  V_{(0,-3)}\otimes \mm_\CC^\ast &\cong V_{(2,-4)}\oplus V_{(-1,-1)}\oplus V_{(-1,-4)}\oplus V_{(-2,-2)}\oplus V_{(1,-5)}\oplus V_{(1,-2)}, \\
  V_{(3,-3)}\otimes \mm_\CC^\ast &\cong V_{(5,-4)}\oplus V_{(2,-1)}\oplus V_{(2,-4)}\oplus V_{(1,-2)}\oplus V_{(4,-5)}\oplus V_{(4,-2)}, \\
  V_{(-3,3)}\otimes \mm_\CC^\ast &\cong V_{(-1,2)}\oplus V_{(-4,5)}\oplus V_{(-4,2)}\oplus V_{(-5,4)}\oplus V_{(-2,1)}\oplus V_{(-2,4)}.
 \end{align*}
 
The adjoint representation of $G$ breaks up into irreducible representations of $H$ as
\[
\gg_\CC \cong 2V_{(0,0)}\oplus V_{(2,-1)}\oplus V_{(-1,2)}\oplus V_{(-1,-1)}\oplus V_{(-2,1)}\oplus V_{(1,-2)}\oplus V_{(1,1)}.
\]
Thus $\gg_\CC$ has precisely two components in common with each of the six tensor products with $\mm_\CC^\ast$ listed above.  Therefore the space of solutions to Equation \eqref{casimir equality} in $L^2(G;E\otimes\mm_\CC^\ast)$ is isomorphic to $12\gg_\CC$.  By Lemma \ref{lemma:deformations}, the space of instanton perturbations is isomorphic to $6\gg$.
\end{proof}

The spaces of solutions to the linearised instanton equations described in the previous theorem are at first sight surprisingly large, given that the expected dimension of the instanton moduli space is zero.  In fact, we can account for all of the perturbations described in this theorem by just two simple observations.

We deal first with the five-dimensional piece $W^\RR_{(1,0)}$ in the case of $\Sp(2)/\Sp(1)\times\U(1)$.  This manifold is the twistor space for $S^4$ and its nearly K\"ahler structure is the canonical nearly K\"ahler structure on the twistor space.  The pull-back of any instanton on a self-dual four-manifold to its twistor space solves the nearly K\"ahler instanton equation on the twistor space (see \cite{Popov:2009nx}).
  The canonical connection on $\Sp(2)/\Sp(1)\times\U(1)$ splits into two connections with holonomy groups $\Sp(1)$ and $\U(1)$; the $\Sp(1)$-part is the pull back of the unique $\Sp(2)$-invariant instanton on $S^4$ with first Pontryagin number 1.  This instanton belongs to a moduli space of dimension five \cite{AHS78}.  Thus the space of deformations of the canonical connection on $\Sp(2)/\Sp(1)\times\U(1)$ with gauge group contained in $\Sp(1)\times\U(1)$ is guaranteed to be at least five-dimensional, and the previous theorem states that this moduli space is in fact exactly five-dimensional.

\begin{corollary}  The perturbations of the canonical instanton on $\CC P^3$ are genuine and are in fact lifts of instantons on $S^4$.
\end{corollary}

The remaining deformations identified in Theorem \ref{thm:explicit deformations} are all isomorphic as representations of the automorphism group $G$ to multiple copies of the Lie algebra $\gg$ of automorphic vector fields on the nearly K\"ahler manifold.  This suggests the existence of an operation that converts automorphic vector fields into instanton perturbations.  Such an operation is identified in the following proposition.

\begin{proposition}
\label{prop:automorphic vector fields}
Let $A$ be an instanton on a principal bundle $\PPP$ over a nearly K\"ahler six-manifold $M$.  Let $X$ be an automorphic vector field for the SU(3) structure and let $\chi$ be a section of $\su(3)M\otimes\Ad_\PPP\subset\Lambda^2M\otimes\Ad_\PPP$ such that $\nabla^{1,A}\chi=0$.  Let $\epsilon_X=\iota_X\chi\in\Gamma(T^\ast M\otimes\Ad_\PPP)$.
Then $\epsilon_X$ solves the infinitesimal instanton equation
\[ \dd^A\epsilon_X \cdot\psi=0. \]
\end{proposition}

\begin{proof}
First we explore the consequences of $X$ being an automorphic vector field.  By definition, the Lie derivatives with respect to $X$ of $g$, $\omega$ and $\Omega$ are zero.  For any section $u$ of $(T^*)^{\otimes p}M$ and any connection $\nabla$ on $TM$ with torsion $T$,
\begin{equation*}
	\mathcal{L}_X u(Y_1,\ldots, Y_p) - \nabla_X u(Y_1,\ldots, Y_p) = \sum_{i=1}^p u(Y_1,\ldots,Y_{i-1},\nabla_{Y_i}X+T(X,Y_i),Y_{i+1},\ldots,Y_p).
\end{equation*}
Suppose that $\nabla$ is a connection with holonomy contained in $\SU(3)$ and let $u=g$.  Then the right hand side describes the natural action of the section $\nabla X+T(X,\cdot)$ of $\End(TM)$ on $g$, while the left hand side of the identity vanishes.  Therefore $\nabla X+T(X,\cdot)$ takes values in the sub-bundle of $\End(TM)$ that fixes $g$ and whose fibre is isomorphic to $\so(6)$.  Similarly, the cases $u=\omega,\Omega$ of the identity tell us that $\nabla X+T(X,\cdot)$ fixes $\omega$ and $\Omega$.  The conclusion then is that if $\nabla$ is any connection with holonomy contained in SU(3) and $X$ is an automorphic vector field for the $\SU(3)$-structure then
\[ \nabla X+T(X,\cdot) \in \Gamma(\su(3)M)\subset\Gamma(\End(TM)). \]

Now we show that $\epsilon_X=\iota_X \chi$ solves the infinitesimal instanton equation.  Given any connection $A$ on $\PPP$ and vector fields $Y,Z$,
\begin{align*}
	\dd^A\epsilon_X(Y,Z)&=Y(\epsilon_X(Z))-Z(\epsilon_X(Y))-\epsilon_X([Y,Z])\\
	&=(\nabla_Y^A\epsilon_X)(Z)+\epsilon_X(\nabla^A_{Y}Z)-(\nabla_Z^A\epsilon_X)(Y)-\epsilon_X(\nabla^A_{Z}Y)-\epsilon_X([Y,Z])\\
	&=(\nabla_Y^A\epsilon_X)(Z)-(\nabla_Z^A\epsilon_X)(Y)+\epsilon_X(T(Y,Z)).
\end{align*}
We choose to use the canonical connection $\nabla=\nabla^1$.  This connection has holonomy contained in $\SU(3)$, so $\nabla^{1,A}\chi=0$ and therefore $\nabla^{1,A}_Y\epsilon_X=\iota_{\nabla^1_YX}\chi$.  Thus
\[ \dd^A\epsilon_X(Y,Z) = \chi(\nabla^1_YX,Z)-\chi(\nabla^1_ZX,Y)+\chi(X,T(Y,Z)). \]
We rewrite the right hand side of this equation as follows:
\begin{multline*}
\dd^A\epsilon(Y,Z) = \chi(\nabla^1_YX+T(X,Y),Z)+\chi(Y,\nabla^1_ZX+T(X,Z)) \\
- \chi(T(X,Y),Z)-\chi(Y,T(X,Z))+\chi(X,T(Y,Z)).
\end{multline*}

The terms on the first line describe the linear action of $\nabla X+T(X,\cdot)$ on the 2-form part of $\chi$.  Since $\chi$ is a section of $\su(3)M\otimes\Ad_\PPP\subset\Lambda^2M\otimes\Ad_\PPP$ and (by the above argument) the endomorphism $\nabla X+T(X,\cdot)$ fixes this subbundle, the terms in the first line also describe a section of this subbundle.

The terms in the second line describe the natural action of the two-form part of $\chi$ on the three-form $P$.  More concretely, if we identify $\chi$ with a section $\tilde{\chi}$ of $\End(TM)\otimes\Ad_\PPP$ such that
\[ g(\tilde{\chi}(Y),Z)=\chi(Y,Z) \quad\forall Y,Z\in\Gamma(TM) \]
then the  second line is equal to
\[ P(\tilde{\chi}(X),Y,Z)+P(X,\tilde{\chi}(Y),Z)+P(X,Y,\tilde{\chi}(Z)) \]
Since the two-form part of $\chi$ belongs to the subspace identified with $\su(3)$ and $\su(3)$ fixes $P$, these terms vanish.  Therefore $\dd^A\epsilon(Y,Z)$ is a section of $\Ad_\PPP\otimes \su(3)\subset \Ad_\PPP\otimes\Lambda^2M$, so solves the infinitesimal instanton equation.
\end{proof}

Proposition \ref{prop:automorphic vector fields} accounts for all of the remaining deformations identified in Theorem \ref{thm:explicit deformations}, as we now briefly explain.

Note first that the curvature $F$ of the canonical connection on any coset space is a parallel section of $\su(3)M\otimes\Ad_\PPP$, so the previous proposition may be applied to $\chi=F$.  The resulting deformations $\epsilon_X=\iota_XF$ have $\dd^A\epsilon_X = \mathcal{L}_XF + [\iota_XA,F]$.  
 Since $F$ is invariant under automorphisms up to gauge, $L_XF=[\lambda_X,F]$ for some infinitesimal gauge transformation $\lambda_X$.  Since $[\iota_XA,F]$ also corresponds to the action of an infinitesimal gauge transformation, we conclude that the deformations in Proposition \ref{prop:automorphic vector fields}  obtained from $F$ are in the direction of the gauge orbit.

In order to apply the proposition we must identify all parallel sections of $\su(3)M\otimes\Ad_\PPP$ with $M=G/H$ and $\PPP$ the SU(3)-structure bundle.  These sections are in bijection with the $H$-invariant elements of the representation $\su(3)\otimes\su(3)$ of $H$, which form a vector space whose dimension equals the sum over all irreducible subrepresentations of $\su(3)$ of the squares of their multiplicities.  The number of irreducible components may be identified from the proof of Theorem \ref{thm:explicit deformations}:
\begin{itemize}
\item
In the case of $G_2/\SU(3)$, the representation $\su(3)_\CC$ of $\SU(3)$ has one irreducible component; it corresponds to the curvature of the canonical connection so Proposition \ref{prop:automorphic vector fields} yields no non-trivial instanton deformations.
\item
In the case of $\SU(2)^3/\SU(2)$, the representation $\su(3)_\CC\cong V_2\oplus V_4$ of $\SU(2)$ has two irreducible components and $\su(3)\otimes\su(3)$ contains two copies of the trivial representation, one in $V_2\otimes V_2$ and one in $V_4\otimes V_4$; the invariant element of $V_2\otimes V_2$ corresponds to the curvature of the canonical connection so the proposition yields for each element of $\gg$ a one-dimensional space of instanton perturbations.
\item
In the case of $\Sp(2)/\Sp(1)\times\U(1)$, the representation $\su(3)_\CC\cong V_{(2,0)}\oplus V_{(0,0)}\oplus V_{(1,3)}\oplus V_{(1,-3)}$ has four irreducible components; the invariant elements of $V_{(2,0)}\otimes V_{(2,0)}$ and $V_{(0,0)}\otimes V_{(0,0)}$ correspond to the $\sp(1)$- and $\uu(1)$-parts of the canonical connection, so the proposition yields for each element of $\gg$ two instanton perturbations.
\item
In the case of $\SU(3)/\U(1)^2$, the representation $\mathfrak{su}(3)_\CC\cong 2V_{(0,0)}\oplus V_{(3,0)}\oplus V_{(-3,0)}\oplus V_{(0,3)}\oplus V_{(0,-3)}\oplus V_{(3,-3)}\oplus V_{(-3,3)}$ of $\U(1)^2$ has six irreducible components of multiplicity one and one irreducible component of multiplicity two.  The invariant elements in $2V_{(0,0)}\otimes 2V_{(0,0)}$ are associated with components of the curvature of the canonical connection, so after discarding them the proposition yields for each element of $\gg$ six instanton perturbations.
\end{itemize}

\appendix 

\begin{center}\Large\textbf{Appendices}\end{center}

\section{The K\"ahler form and the complex $(3,0)$-form.}
\label{sec:omegaOmega}
In this appendix the identities \eqref{omegas} are proven.

The proof of the first of these requires careful track to be kept of minus signs, so let us first point out that, by analysis of eigenvalues (using Lemma \ref{evals PQ}),
\[ \ast Q\cdot\ast Q = -3 + 2Q. \]
This equation is consistent with the identification $\omega=\ast Q$ but not with $\omega=-\ast Q$, because $\ast Q\cdot\ast Q = -\|\ast Q\|^2 + 2\ast Q\wedge\ast Q$ and $\omega\wedge\omega=2\ast \omega$ in dimension six.

The full proof of the first equality now follows.  The definition of $\omega$ is
\[ \frac18 \Tr_S (\omega\cdot u\cdot v) = -g(u,Jv) \quad\forall u,v\in V. \]
Now
\begin{align*}
 \frac18 \Tr_S (\ast Q\cdot u\cdot v) &= -\Tr_S (\psi\otimes\psi^T \cdot \Vol \cdot u\cdot v) \\
                                      &= -(\psi, \Vol \cdot u \cdot v \cdot \psi) \\
                                      &= (\psi, u\cdot\Vol\cdot v\cdot\psi) \\
                                      &= (\psi, u\cdot Jv \cdot\psi) \\
                                      &= -(\psi, g(u,Jv) \cdot\psi) \\
                                      &= - g(u,Jv).
\end{align*}
So $\ast Q=\omega$ as desired.

To prove the second equality, it suffices to prove that $(v-\ii Jv)\lrcorner (P+\ii*P)=0$ for all cotangent vectors $v$.   Using Lemma \ref{lemma:algebra}, this statement is equivalent to proving that 
$\{v-\ii Jv,P+\ii*P\}=0$
for all vectors $v$.  Having moved the statement to the Clifford algebra, we can now use the definition of $J$, namely $Jv\cdot \psi=\Vol\cdot v\cdot \psi$.
Since $\Vol\cdot v=*v$, $Jv$ acts on $\psi\otimes \psi^T$ just as $*v$ does, and from this it can be shown that $\{Jv, P\} = \{*v, P\}$ and $\{Jv, *P\} = \{*v, *P\}$.

Moreover, for any three-form $\alpha$, we have $v\cdot \alpha =-(*v)\cdot (*\alpha)$ and $\alpha\cdot v=-(*\alpha)\cdot (*v)$.  Thus
we have
\begin{align*}
	\{v-\ii Jv,P+\ii*P\}&= \{v,P\} + \{Jv,*P\}+\ii\bigl(\{v,*P\}-\{Jv,P\}\bigr)\\
	&=\{v,P\} + \{*v,*P\}+\ii\bigl(\{v,*P\}-\{*v,P\}\bigr)\\
	&=\{v,P\} -\{v,P\}+\ii\bigl(\{v,*P\}+\{*^2v,*P\}\bigr)\\
	&=\{v,P\} -\{v,P\}+\ii\bigl(\{v,*P\}-\{v,*P\}\bigr)\\
	&=0,
\end{align*}
as desired.
\section{\texorpdfstring{$\su(3)\subset\g_2$}{su(3) inside g2}} 
\label{sec:_g_2}
This appendix describes very succinctly but explicitly the embedding $\sll_3\C\subset (\gg_2)_\C$ and exhibits the results needed in the main part of the paper.

We follow the notation and descriptions given in \cite[\S 22.1 to \S 22.3]{FultonHarris}.
The Lie algebra $(\gg_2)_\CC$ is spanned by the 14 elements $H_1,H_2,X_1,\ldots, X_6,Y_1,\ldots,Y_6$. The Cartan subalgebra is spanned by $H_1,H_2$, while   the $X_i$ belong to the root space of the root $\alpha_i$ and the $Y_i$ belong to the root space of the root $\beta_i=-\alpha_i$.
In the Cartan subalgebra, one defines $H_3:=H_1+3H_2$, $H_4:=2H_1+3H_2$, $H_5:=H_1+H_2$, $H_6:=H_1+2H_2$
to simplify the multiplication table.  They have the desirable property that
$H_i=[X_i,Y_i]$, $[H_i,X_i]=2X_i$, $[H_i,Y_i]=-2Y_i$.
The  positive roots are written in terms of the simple roots $\alpha_1,\alpha_2$ as $\alpha_3=\alpha_1+\alpha_2,
	\alpha_4=2\alpha_1+\alpha_2,
	\alpha_5=3\alpha_1+\alpha_2,
	\alpha_6=3\alpha_1+2\alpha_2$.
The simple roots satisfy	$\alpha_1(H_1)=2,\alpha_1(H_2)=-1,
	\alpha_2(H_1)=-3,\alpha_2(H_2)=2$.

There is a copy $\hh$ of \[\sll_3\C=\Span\{E_{11}-E_{22},E_{22}-E_{33},E_{12},E_{21},E_{23},E_{32},E_{13},E_{31}\}\] in $(\g_2)_\C$ given by
\[\hh:=\Span\{H_5,H_2,X_5,Y_5,X_2,Y_2,X_6,Y_6\}.\]
The term by term identification of the basis element is an isomorphism of Lie algebras.

Let $W$ be the standard representation of $\sll_3\C$. 
Then we have the orthogonal decomposition
\[(\g_2)_\C = \hh\oplus W\oplus W^*\]
as representation of $\sll_3\C$. So
$\tr_{(\g_2)_\C}(\ad(H_i)\circ \ad(H_j))=\tr_{\hh}(\ad(H_i)\circ \ad(H_j)) + 2\tr(H_iH_j)$.
One can thus compute $B(H_i,H_j)=-\frac1{12}\tr_{(\g_2)_\C}(\ad(H_j)\circ \ad(H_j))$ with $i,j\in\{5,2\}$ to obtain (once relabelling $5$ into $1$) Equation \eqref{eqn:Bsu3}.

The fundamental weights of $(\g_2)_\C$ are
$\lambda_1=2\alpha_1+\alpha_2$ and $\lambda_2=3\alpha_1+2\alpha_2$.
They are dual to the basis $H_1,H_2$. In terms of $\hh=\sll_3\C$, we have
$H_1=H_5-H_2=\diag(1,-2,1)$ and $H_2=\diag(0,1,-1)$, as mentioned on page \pageref{H1H2G2}. 
We can compute easily $\tr_{(\g_2)_\C}(\ad(H_i)\circ \ad(H_j))$ for $i,j\in\{1,2\}$ to obtain Equation \eqref{eqn:BG2}.

Note also 
\begin{equation}\label{eqn:deltaG2}
	\delta=\frac{\sum_{i=1}^6\alpha_i}2=5\alpha_1+3\alpha_2=\lambda_1+\lambda_2,
\end{equation}as claimed on page \pageref{deltaG2}.

\section{\texorpdfstring{$\su(2)\subset \su(2)^3$}{su(2) inside su(2)+su(2)+su(2)}} 
\label{sec:_su_2}
The Lie algebra $\hh=\su(2)$ has a basis $J_1,J_2,J_3$ such that $[J_i,J_j]=\epsilon_{ijk}J_k$.   Then $\tr_\hh(\ad(J_i)\circ\ad(J_j))=-2\delta_{ij}$.
 A basis for the diagonal $\su(2)\subset \su(2)\oplus\su(2)\oplus\su(2)$ is given by $I_i=(J_i,J_i,J_i)$.  
The elements $K_i=(J_i,-J_i,0)$ and $L_i=(J_i,J_i,-2J_i)$ span the orthogonal complement $\mm$.
Now $\Span(K_1,K_2,K_3)$ and $\Span(L_1,L_2,L_3)$ are, as representations, copies of $\hh$.
   Thus
$B(J_i,J_j)=-\sfrac1{12}3\tr_\hh(\ad(J_i)\circ\ad(J_j))=\sfrac12\delta_{ij}$.

Given that a vector $(X,Y)\in\C^6$ represents the element
\[(X_1\frac{K_1}{\sqrt2}+X_2\frac{K_2}{\sqrt2}+X_3\frac{K_3}{\sqrt2},Y_1\frac{L_1}2+Y_2\frac{L_2}2+Y_3\frac{L_3}2)\in V_2\oplus V_2,\]  \cite[p.12]{Butruille} gives  the expression
\[J\colon (X,Y)\mapsto \frac1{\sqrt3}(2Y-X,-2X+Y)\]
for the  almost complex structure.
In view of this definition, we have
\[J(\sqrt{2}K_i+L_i)=\sqrt{6}K_i\text{ and }J(\sqrt{6}K_i)=-(\sqrt{2}K_i+L_i).\]

Given the complex ordered basis $\BB=(\sqrt{2}K_1+L_1,\sqrt{2}K_2+L_2,\sqrt{2}K_3+L_3)$ of $\mm$, we have
\begin{align*}
	{}_\BB[\ad(I_1)]_\BB&=\vect{0&0&0\\ 0&0&-1\\ 0&1&0},\quad
	{}_\BB[\ad(I_2)]_\BB&=\vect{0&0&1\\ 0&0&0\\ -1&0&0},\quad 
	{}_\BB[\ad(I_3)]_\BB&=\vect{0&-1&0\\ 1&0&0\\ 0&0&0}.
\end{align*}
This action provides the homomorphism $\su(2)\subset \su(3)$ necessary to understand Equation \eqref{eqn:su3V2V4}.


\subsection*{Acknowledgements.} 
Important face-to-face discussions in the redaction of this paper took place while both authors were visiting the Newton Institute for Mathematical Sciences in Cambridge for the \emph{Metric and Analytic Aspects of Moduli Spaces} programme.  Many thanks to the Institute staff and the programme organizers for a stimulating intellectual atmosphere.    Part of this work was accomplished during BC's sabbatical at the Perimeter Institute for Theoretical Physics.  Research at Perimeter Institute is supported by the Government of Canada through Industry Canada and by the Province of Ontario through the Ministry of Research \& Innovation.  BC supported by an NSERC Discovery Grant.  Sincere thanks also go to the two anonymous referees for the time they took reviewing this work and for their useful comments.

\bibliographystyle{abuser}
\bibliography{NKinst,Higher-Dim-Instantons}

\end{document}